\documentclass[11pt,a4paper]{article}

\usepackage{amsmath,amsfonts,amssymb}
\usepackage{color, graphicx, soul}
\usepackage{mathtools}
\usepackage{xargs}
\usepackage{enumerate}
\usepackage[colorinlistoftodos,prependcaption,textsize=footnotesize]{todonotes}
\newcommandx{\attn}[2][1=]{\todo[linecolor=red,backgroundcolor=blue!25,bordercolor=red,#1]{#2}}
\newcommandx{\other}[2][1=]{\todo[linecolor=OliveGreen,backgroundcolor=OliveGreen!25,bordercolor=OliveGreen,#1]{#2}}
\newcommandx{\thiswillnotshow}[2][1=]{\todo[disable,#1]{#2}}
\usepackage{hyperref}
\hypersetup{
	colorlinks,
	linkcolor={red!50!black},
	citecolor={blue!50!black},
	urlcolor={blue!80!black}
}

\newtheorem{theorem}{Theorem}

\newtheorem{proposition}{Proposition}
\newtheorem{remark}[theorem]{Remark}

\newcommand{\beq}{\begin{equation}}
\newcommand{\eeq}{\end{equation}}
\newcommand{\beqa}{\begin{eqnarray}}
\newcommand{\eeqa}{\end{eqnarray}}
\newcommand{\beqas}{\begin{eqnarray*}}
\newcommand{\eeqas}{\end{eqnarray*}}
\newcommand{\bi}{\begin{itemize}}
\newcommand{\ei}{\end{itemize}}

\newcommand{\PSDcone}[1]{{\mathcal{S}^{#1}_+}}
\renewcommand{\S}{\mathcal{S}}      
\renewcommand{\Re}{\mathbb{R}}
\def\qed{\ifhmode\unskip\nobreak\fi\ifmmode\ifinner\else\hskip5pt\fi\fi
  \hbox{\hskip5pt\vrule width5pt height5pt depth1.5pt\hskip1pt}}

\def\QED{\ifhmode\unskip\nobreak\fi\ifmmode\ifinner\else\hskip5pt\fi\fi
  \hbox{\hskip5pt\vrule width5pt height5pt depth1.5pt\hskip1pt}}
\newcommand{\stdCone}{ \ensuremath{\mathcal{K}}}
\def\Id{{I_d}}
\def\Ip{{I_p}}

\usepackage{cite}

\usepackage{hyperref}
\hypersetup{
	colorlinks,
	linkcolor={red!50!black},
	citecolor={blue!50!black},
	urlcolor={blue!80!black}
}
\renewcommand{\S}{\mathcal{S}} 
\renewcommand{\Re}{\mathbb{R}} 
\newcommand{\iPSDcone}[1]{{\mathcal{S}^{#1}_{++}}}
\def\Id{I_d}
\def\Ip{I_p}

\def\av{{v_a}}
\newcommand{\norm}[1]{\lVert{#1}\rVert}
\begin{document}

\title{
A Limiting Analysis on Regularization of Singular SDP and its Implication to Infeasible Interior-point Algorithms
\thanks{The first author is supported in part by MEXT Grant-in-Aid for Scientific Research (B)18H03206,
the second author is supported in part by the same grant and MEXT Grant-in-Aid for Young Scientists 19K20217,
the third author is supported in part by MEXT Grant-in-Aid for Scientific Research (C)17K00031 and
the same grant for Scientific Research (B)20H04145, and 
the fourth author is supported in part by MEXT Grant-in-Aid for Young Scientists 20K19748 and the same grant for Scientific Research (B)20H04145.
}  \\ \ \ \ \\
}


\author{Takashi Tsuchiya\footnote{National Graduate Institute for Policy Studies, 7-22-1 Roppongi, Minato-ku, Tokyo 106-8677 Japan, 
              e-mail: {tsuchiya@grips.ac.jp}}           %
\and
         Bruno F.~Louren\c{c}o\footnote{The Institute of Statistical Mathematics, Midori-cho 10-3, Tachikawa, 190-8562 Tokyo Japan, e-mail: {bruno@ism.ac.jp}} \and Masakazu Muramatsu\footnote{The University of Electro-Communications, 
1-5-1 Chofugaoka, Chofu, Tokyo 182-8585 Japan, e-mail: MasakazuMuramatsu@uec.ac.jp } \and Takayuki Okuno\footnote{  Center for Advanced Intelligence Project, RIKEN, 1-4-1 Nihonbashi, Chuo-ku, Tokyo 103-0027 Japan.
email: {takayuki.okuno.ks@riken.jp}       }
}

\date{Revised: March, 2022 \\ (Original: December 2019, Revised: May 2021)}

\maketitle
\begin{abstract}
We consider primal-dual pairs of semidefinite programs and assume that they are singular, i.e., both primal and dual are either weakly feasible or
weakly infeasible.  Under such circumstances, 
strong duality may break down and the primal and dual might have a nonzero duality gap.
Nevertheless, there 
are arbitrary small perturbations to the problem data which would make them strongly feasible thus 
zeroing the duality gap.
In this paper, we conduct an asymptotic  analysis of the optimal value 
as the perturbation for regularization is driven to zero.
Specifically, we fix two positive definite matrices, $I_p$ and $I_d$, say, (typically the identity matrices), and regularize
the primal and dual problems by shifting their associated affine space by $\eta I_p$ and $\varepsilon I_d$, respectively, to recover interior feasibility of both problems, where $\varepsilon$ and $\eta$ are positive numbers.
Then we analyze the behavior of the optimal value of the regularized problem when the perturbation is reduced to zero keeping the ratio between $\eta$ and
$\varepsilon$ constant. 
A key feature of our analysis is that no further assumptions such as compactness or constraint qualifications are ever made.
It will be shown that the optimal value of the perturbed problem converges to a value between the primal and dual optimal values of the original problems.  
Furthermore, the limiting optimal value changes ``monotonically'' from the primal optimal value to the dual optimal value as a function of $\theta$, 
if we parametrize $(\varepsilon, \eta)$ as $(\varepsilon, \eta)=t(\cos\theta, \sin\theta)$ and let $t\rightarrow 0$.
Finally, the analysis leads us to the relatively surprising consequence that some representative 
infeasible interior-point algorithms for SDP generate sequences converging to a number 
between the primal and dual optimal values, even in the presence of a nonzero duality gap. 
Though this result is more of theoretical interest at this point, it might be of some value in the development of infeasible interior-point algorithms that can handle  singular problems.

\medskip\noindent
{\bf keywords}: Semidefinite programs, singular problems, nonzero duality gaps, perturbation, regularization, infeasible interior-point algorithms
\end{abstract}

\section{Introduction}

{
Strong feasibility of primal and dual problems is a standard regularity condition 
in convex optimization, e.g.,  \cite{Luo97dualityresults}, \cite[Chapter~3]{Renegar_book_2001}.}  
Once this condition is satisfied,    
powerful algorithms such as interior-point algorithms and the ellipsoid algorithm 
can be applied to solve them efficiently, at least in theory.  
On the other hand,  if a problem at hand does not satisfy this condition, it can be much harder to solve.
{For instance, the problem may have a positive duality gap.}
Due to the advance of techniques of optimization modelling, there are many problems which do not satisfy primal-dual 
strong feasibility by nature.  

{A first attempt to apply interior-point algorithms to such problems would be to
perturb the problem to recover strong feasibility at both sides, i.e., ``regularization.''  But it is not clear how this perturbation affects
the optimal value.  
In this paper, we focus on semidefinite programs (SDP) and conduct an asymptotic analysis 
of the optimal value function when the problem is perturbed slightly to recover primal-dual strong feasibility.   
The analysis is general enough to be applicable to any { ill-behaved} problems without assuming 
constraint qualifications, and has interesting implications to the convergence theory of interior-point algorithms.}

{
It is known that every SDP falls into one of the four statuses: strongly feasible, weakly feasible, weakly infeasible and
strongly infeasible, e.g., \cite{Luo96dualityand}.} { Difficult} situations like positive duality gap may occur when
the problem is either weakly feasible or weakly infeasible.  
{We may call such problems ``singular.''}

{ A standard method to deal with singular problems 
in semidefinite programming and general conic convex programming is facial reduction
\cite{borwein_facial_1981,Borwein1981495,csw13,DW17,sturm_error_2000,WM13,LP17,sremac2017complete}.}
This approach recovers strong feasibility by finding the minimal face containing the feasible region.
{ While many of the earlier papers on facial reduction focused on weakly feasible problems, it is relatively recent that weak 
infeasibility is analyzed in this context \cite{lourenco_muramatsu_tsuchiya, LP17,GoodPataki2017}.   
Along this line of developments, 
the paper \cite{LMT21} showed that any SDP can be solved completely just by calling an interior-point oracle polynomially many times by using facial reduction, where
the interior-point oracle is an idealized interior-point algorithm which returns a primal-dual optimal solutions given a primal-dual 
strongly feasible SDP.  
In the context of SDPs with positive duality gaps, Ramana 
developed an extended Lagrangian dual SDP 
for which strong duality always holds \cite{Ramana95anexact}.  
Later it was shown in \cite{ramana_strong_1997} this dual problem is strongly related to facial reduction, see also \cite{pataki_strong_2013}.}

Implementation of a facial reduction algorithm is subtle and not easy, being vulnerable to rounding errors.
Nevertheless, it is worth mentioning that there are several recent works focused on practical issues regarding facial reduction or on heuristics based on facial reduction \cite{PP17,PFA17,Fr16,YPT17}.

{ So far, we have discussed approaches based on (or related to) facial reduction in order to deal with singular SDPs. 
 Unrelated to that, the paper \cite{Liu2019} considered an application of the Douglas-Rachford algorithm to the analysis of pathological behavior in SDPs. Interestingly, they show it is sometimes possible to identify the presence of positive duality gaps by observing whether certain sequences converge to $0$ or to $\infty$, see \cite[Figure~1, Sections~2.8 and 2.9]{Liu2019}.
  
As mentioned previously, in this paper we will consider yet another approach for analyzing singular SDPs: \emph{regularization}.}
The idea is to perturb the problem slightly to recover strong feasibility on both primal and dual sides.  
Once strong feasibility is recovered, we may, say, apply interior-point algorithms to the regularized problems.
However, the resulting approximate optimal solution is not guaranteed to be close to the optimal solution to 
the original problem, though intuitively we might expect or hope so.  In particular, if we 
consider a SDP problem with a finite and nonzero duality gap, 
it is not clear what happens with the optimal value and the optimal solutions of the regularized problem as functions of the perturbation
when the perturbation is reduced to zero.   

Analyzing this problem is one of the main topics of the current paper.
We consider  primal and dual pairs of semidefinite programs and assume they are singular i.e., either weakly feasible or
weakly infeasible (see Section~\ref{sec:setup} for definitions).  Under these circumstances, there are arbitrarily small perturbations which make the perturbed pair primal-dual strongly feasible. Then, we fix two positive definite matrices, and shift the associated affine spaces of the primal and dual slightly in the direction of these matrices 
so that the perturbed problems have interior feasible solutions. Under this setting, we analyze the behavior of the optimal value of the perturbed problem when the perturbation is reduced to zero while keeping the proportion.  

First, we demonstrate that, if perturbation is added only to the primal problem to recover strong feasibility, 
then the optimal value of the 
perturbed  problem converges to the dual optimal value as the perturbation is reduced to zero, 
even in the presence of nonzero duality gap.  
An analogous proposition holds for the dual problem.  We derive them as a significantly 
simplified version of  the classical asymptotic strong duality theorem (see, for instance, \cite{Duffin1957,bams/1183530289,BS00,Renegar_book_2001,Luo96dualityand,Luo97dualityresults} and Chapter 2 of \cite{sturm_high}).

Then we analyze the case where perturbation is added to \emph{both}  primal and dual sides of the problem.  We will demonstrate that
in that case the limiting 
optimal value of the perturbed problems converges to
a value between the primal and dual optimal values of the original problem even in the presence of nonzero duality gap.  
The limiting optimal value is a function of the relative weight 
of primal and dual perturbations, and reduces monotonically from the primal optimal value to the dual optimal value as the relative weight shifts from the dual side to the primal side.

The result provides an interesting implication to the behavior of infeasible interior-point algorithms applied to general SDPs
\cite{Helmberg:1996:IPM,Kojima:1997:IPM,Monteiro:1997:PDP,nesterov_todd_1,Todd:1998:NTD,Zhang:1998:ESP,ipm:Potra16}. 
In particular, we pick up two well-known polynomial-time infeasible interior-point algorithms by Zhang \cite{Zhang:1998:ESP} and 
Potra and Sheng \cite{ipm:Potra16}, and prove the following (see Theorems~\ref{5.1} and \ref{5.2}): 

\noindent
\begin{enumerate}
\item
If neither the primal nor the dual are strongly infeasible then:
\begin{enumerate}
	\item the algorithms always generate sequences $(X^k,S^k,y^k)$ that are asymptotically primal-dual feasible and such that the ``duality gap'' $X^k\bullet S^k$
 converges to zero.
	\item the sequence of modified (primal and dual) objective values converges to a number in $[\theta_D, \theta_P]$, where $\theta_P$ and 
	$\theta_D$ are the primal optimal value and the dual optimal value, respectively. 
\end{enumerate}
\item Otherwise (i.e., if either the primal or the dual is strongly infeasible), the algorithms fail to generate a sequence such 
that the duality gap $X^k\bullet S^k$
converges to zero. (Needless to say, there is no way to 
generate an asymptotically primal-dual feasible sequence in this case.) 
\end{enumerate}
One implication of the result above is that, at least in theory, these interior-point algorithms generate sequences converging
to the optimal value as long as strong feasibility is satisfied at one side of the problem.
Furthermore, even in the presence of a finite duality gap,
they still generate sequences converging to values between the primal and dual optimal values.  It is also worth mentioning 
that our analysis shows that, by setting appropriate initial iterates, it is possible 
to control how close the limit value will be to the primal or the dual optimal values.  

Though this result is more of theoretical interest, this might be of some value if one wants to solve mixed-integer SDP (MISDP) through branch-and-bound and linear SDP relaxations.
As discussed in \cite{GPU18}, it is quite possible that the relaxations eventually fail to satisfy strong feasibility at  least one of the sides of the problem.

Nevertheless, the solutions obtained by 
the infeasible interior-point methods described above can still be used as bounds 
to the optimal values of the relaxed linear SDPs regardless of regularity assumptions or constraint qualifications (at least in theory).

This paper is organized as follows.  In Section~\ref{sec:main}, we describe our main results.  Section~\ref{sec:prel} is a preliminary section 
where we review asymptotic strong duality, infeasible interior-point algorithms, and semialgebraic geometry.
In Section~\ref{sec:proof}, we develop a main analysis when both primal and dual problems are perturbed.  In Section~\ref{sec:app}, we apply the developed result 
to an analysis of the infeasible primal-dual algorithms.  In Section~\ref{sec:ex}, illustrative instances will be presented.

\section{Main Results}\label{sec:main}


In this section, we introduce our main results after providing the  setup and some preliminaries. 
We also review existing related results.

\subsection{Setup and Terminology}\label{sec:setup}

First we introduce the notation.
The space of $n\times n$ real symmetric matrices 
will be denoted by $\S^n$. We denote the cone of $n\times n$ real symmetric positive semidefinite matrices and
the cone of $n\times n$ real symmetric positive definite matrices by $\PSDcone{n}$ and $\iPSDcone{n}$.
For $U,V \in \S^n$, we define the inner product $U\bullet V$ as $\sum U_{ij} V_{ij}$, and we use 
 $U \succeq 0$ and $U\succ 0$ to denote
that $U \in \PSDcone{n}$ and $U \in \iPSDcone{n}$, respectively. 
The $n\times n$ identity matrix is denoted by $I$.
We denote the Frobenius norm and the operator norm by $\norm{X}_{F}$ and $\norm{X}$. For $v \in \Re^k$, we denote by $\norm{v}$ its Euclidean norm.

In this paper, we deal with the following standard form primal-dual semidefinite programs
\begin{align*}
{\bf P:}&\ \ \ \min_{X} \ C \bullet X \ \ \hbox{s.t.}\ A_i \bullet X = b_i,\ i=1, \ldots, m, X\succeq 0\\
{\bf D:}&\ \ \ \max_{y,S} \ b^T y \ \ \ \ \hbox{s.t.} \ C - \sum_{i=1}^m A_i y_i  = S,\ S\succeq 0,
\end{align*} 
where $C$, $A_i, i = 1, \ldots, m$, $X$, $S$ are real symmetric $n\times n$ matrices and $y \in \Re^m$.    
For ease of notation, we define the mapping $A$ from $\S^n$ to $\Re^m$:
\begin{equation}\label{eq:AY}
A(Y)  \equiv (A_1\bullet Y, \ldots, A_m\bullet Y),
\end{equation}
and introduce 
\[\mathcal{V} \equiv \{X \in \S^n \mid   A_i \bullet X = b_i,\ i=1, \ldots, m \} =\{X\in \S^n \mid A(X)=b\}.\]
We denote by  
$v({\bf P})$ and $v({\bf D})$ the optimal values of  {\bf P} and {\bf D}, respectively.
We use analogous notation throughout the paper to denote the optimal value of an optimization problem.
For a maximization problem, the optimal value $+\infty$ means that the optimal value is unbounded above
and the optimal value $-\infty$ means that the problem is infeasible.
For a minimization problem, the optimal value $-\infty$ means that the optimal value is unbounded below
and the optimal value $+\infty$ means that the the problem is infeasible.

It is well-known that $v({\bf P}) = v({\bf D})$ holds under suitable regularity conditions, although, in general, we might have   $v({\bf P}) \neq v({\bf D})$, i.e.,  the problem may have a nonzero duality gap.  We also note that $v({\bf P})$ and $v({\bf D})$ might not be necessarily attainable.

In general,  ${\bf P}$ is known to be in one of the following four different mutually exclusive status (see \cite{Luo97dualityresults}). 

\begin{enumerate}
	\item Strongly feasible: there exists a positive definite matrix satisfying the constraints of ${\bf P}$, i.e., $\mathcal{V} \cap \iPSDcone{n} \neq \emptyset$. This is the same as Slater's condition.
	\item Weakly feasible: ${\bf P}$ is feasible but not strongly feasible, i.e., $\mathcal{V} \cap \iPSDcone{n} =\emptyset$ but $\mathcal{V} \cap \PSDcone{n} \neq \emptyset$.
	\item Weakly infeasible: ${\bf P}$ is infeasible but the distance between $\PSDcone{n}$ and the affine space $\mathcal{V}$ is zero, i.e., 
	 $\mathcal{V} \cap \PSDcone{n} = \emptyset$ but the zero matrix  belongs to the closure of $\PSDcone{n}-\mathcal{V}$.
	\item Strongly infeasible: ${\bf P}$ is infeasible but not weakly infeasible. Note that this includes the case where $\mathcal{V} = \emptyset$.
\end{enumerate}
The status of ${\bf D}$ is defined 
analogously by replacing  
$\mathcal{V}$ by  the affine set \[\{S \in \S^n \mid \exists y \in \Re^m, C - \sum_{i=1}^m A_i y_i  = S  \}.\] 
We say that a problem is {\it asymptotically feasible} if it is either feasible or weakly infeasible. As a reminder, we say that a problem is \emph{singular} if it is either weakly feasible or weakly infeasible.


\subsection{Main Results}\label{subsec2.2}

Now we introduce the main results of this paper.  We say that a problem is {\em asymptotically primal-dual feasible}
(or {\em asymptotically pd-feasible}, in short)
if both {\bf P} and {\bf D} are asymptotically feasible.  Evidently, the problem is asymptotically pd-feasible 
if and only if  both {\bf P} and {\bf D} are feasible or weakly infeasible.  The analysis in this paper is conducted mainly
under this condition.  

Note that asymptotic pd-feasibility is a rather weak condition.
Many difficult situations such as finite nonzero duality gaps and weak infeasibility of both {\bf P} and {\bf D}
are covered under this condition. 
Furthermore, since strong infeasibility can be detected by solving auxiliary SDPs that are both primal and dual strongly 
feasible (see \cite{LMT15}), checking whether a given problem is asymptotically pd-feasible or not   can also  be checked by solving SDPs that 
are primal and dual strongly feasible.

We consider the following primal-dual pair {\bf P}($\varepsilon,\eta$) and ${\bf D}(\varepsilon,\eta)$
obtained by perturbing {\bf P} and {\bf D} 
with two positive definite matrices $\Ip$ and $\Id$ and two nonnegative parameters $\varepsilon$ and $\eta$:

\begin{equation}\label{Pptbd}
{\bf P}(\varepsilon,\eta):\ \ \ \min \ (C + \varepsilon \Id)\bullet X \ \ \hbox{s.t.}\ A_i \bullet X = b_i+ \eta A_i\bullet \Ip,\ i=1, \ldots, m,\ X\succeq 0,
\end{equation}
and 
\begin{equation}\label{Dptbd}
{\bf D}(\varepsilon,\eta):\ \ \ \max \sum_{i=1}^m (b_i + \eta A_i\bullet \Ip)y_i \ \ \hbox{s.t.} \ C - \sum_{i=1}^m A_i y_i + \varepsilon \Id = S,\ \ \ S\succeq 0.
\end{equation}
Using \eqref{eq:AY}, we have
\begin{equation}\label{Pptbdalt}
{\bf P}(\varepsilon,\eta):\ \ \ \min \ (C + \varepsilon \Id)\bullet X \ \ \hbox{s.t.}\ A(X) = b+ \eta A(\Ip),\ X\succeq 0.
\end{equation}
While $\Ip$ and $\Id$ represent the direction of perturbation, $\varepsilon$ and $\eta$ represent the amount of perturbation.
In particular, we could take, for example, $\Ip = \Id = I$, where $I$ is the $n\times n$ identity matrix.
{ We note that the perturbed pair \eqref{Pptbd} and \eqref{Dptbd} was used in the study of infeasible interior-point algorithms  \cite{ipm:Potra16} and facial reduction\cite{sremac2017complete}.}

If the problem is asymptotically pd-feasible, ${\bf D}(\varepsilon,\eta)$ is strongly feasible for any $\varepsilon > 0$ and
${\bf P}(\varepsilon,\eta)$ is strongly feasible for any $\eta>0$.  To see the strong feasibility of ${\bf P}(\varepsilon,\eta)$, 
we observe that there always exists $\widetilde X\succeq -\eta \Ip/2$ satisfying 
$A_i\bullet \widetilde X = b_i, i=1, \ldots, m$, since {\bf P} is weakly infeasible or feasible.  Then, we see that the matrix $X=\widetilde X+\eta \Ip$ is positive definite and a feasible solution to ${\bf P}(\varepsilon,\eta)$.  
We emphasize that the primal-dual pair ${\bf P}(\varepsilon,\eta)$ and  
${\bf D}(\varepsilon,\eta)$ is a natural and possibly one of the simplest regularizations of ${\bf P}$ and ${\bf D}$ which ensures primal-dual strong feasibility under perturbation.

We define $v(\varepsilon,\eta)$ to 
be the common optimal value of ${\bf P}(\varepsilon,\eta)$ and ${\bf D}(\varepsilon,\eta)$ if they coincide.
If the optimal values differ, $v(\varepsilon,\eta)$ is not defined.
Suppose that {\bf P} and {\bf D} are asymptotically pd-feasible.
In this case, from the the duality theory of convex programs, the function $v(\varepsilon,\eta)$ has the following properties:
\begin{enumerate}
\item $v(\varepsilon,\eta)$ is finite if $\varepsilon >0$ and $\eta>0$.
\item $v(\varepsilon,0)$ is well-defined as long as $\varepsilon > 0$ and it takes the value $+\infty$ if {\bf P} is infeasible.
\item $v(0,\eta)$ is well-defined as long as $\eta > 0$ and it takes the value $-\infty$ if {\bf D} is infeasible.
\item $v(\varepsilon,\eta)$ may not be defined at $(0,0)$.  This is because ${\bf P}={\bf P}(0,0)$ and ${\bf D}={\bf D}(0,0)$ may have 
different optimal values, i.e.,  {\bf P} and {\bf D} may have a nonzero duality gap.
\end{enumerate}
Therefore, although the regularized pair ${\bf P}(\varepsilon,\eta)$ and ${\bf D}(\varepsilon,\eta)$ satisfies primal-dual strong feasibility
if $\varepsilon > 0$ and $\eta > 0$,
it is not clear whether this is actually useful in solving SDP under { notorious} situations such as the presence of nonzero duality gaps.  
This is precisely one of the main topics of this paper: an analysis on the behavior of the regularized problems without
imposing any restrictive assumption. 

In this context, it is worth mentioning that the following asymptotic strong duality results
\[
\hbox{(i)}\ \lim_{\varepsilon \downarrow 0}v(\varepsilon, 0) =
\lim_{\varepsilon \downarrow 0}v({\bf D}(\varepsilon, 0)) =
 v({\bf P})\ \hbox{under\ dual\ asymptotic\ feasibility}
\]
and
\[
\hbox{(ii)}\ \lim_{\eta \downarrow 0}v(0,\eta)=\lim_{\eta \downarrow 0}v({\bf P}(0,\eta)) = v({\bf D})
\hbox{\ under\ primal\ asymptotic\ feasibility}
\]
are obtained as corollaries of the classical asymptotic strong duality theorem established in the 1950's and 1960's  
\cite{Duffin1957,bams/1183530289}.  { This theory} received renewed attention with the emergence of conic linear programming; see, for instance,
\cite{BS00,Renegar_book_2001,Luo96dualityand,Luo97dualityresults} and Chapter 2 of \cite{sturm_high}.  
We will prove (i) and (ii) in the next section, see {Theorem~\ref{3.2}}. In comparison with the classical asymptotic strong duality theorem, Theorem~\ref{3.2} considers a smaller perturbation space.

%

Now we are ready to describe the main results.  They are developed to interpolate between (i) and (ii).
The first result is the following theorem. 
\begin{theorem}\label{2.1}
Let $\alpha \geq 0$, $\beta \geq 0$ and $(\alpha,\beta)\not=(0,0)$.  If the problem is asymptotically pd-feasible, then   
$\lim_{t\downarrow 0} v(t\alpha, t\beta)$ exists. 
\end{theorem}
Here we remark that Theorem~\ref{2.1} includes the case where the limit is $\pm \infty$.
Theorem~\ref{2.1} implies that the limit of the optimal value of the perturbed system exists but 
it is a function of the direction used to approach $(0,0)$.  For $\theta \in [0,\pi/2]$, let us consider the function
\[
\av(\theta) \equiv \lim_{t\downarrow 0}v(t\cos \theta,t\sin\theta),
\]
which is the limiting optimal value of $v(\cdot)$ when it approaches zero along the direction making an  angle of $\theta$ with the $\varepsilon$ axis.  With that, $\av(0)$ and $\av(\pi/2)$ are the special cases corresponding to dual-only perturbation and primal-only perturbation, respectively.  So we abuse notation slightly and define 
\begin{equation}\label{eq:vadvap}
\av({\bf D}) \equiv \av(0) \quad \text{ and }\quad  \av({\bf P}) \equiv \av(\pi/2).
\end{equation}
Below is our second main result.

\begin{theorem} \label{2.2} If the problem is asymptotically pd-feasible, the following statements hold.

\begin{enumerate}
\item $\av(0) = \av({\bf D})=v({\bf P})$ and $\av(\pi/2) =\av({\bf P})= v({\bf D})$.
\item $\av(\theta)$ is monotone decreasing in $[0, \pi/2]$,
and is continuous in $(0,\pi/2)$.
\end{enumerate}
\end{theorem}
Theorem \ref{2.2} is proved by using Theorem \ref{4.2} which establishes monotonicity and convexity of 
$\lim_{t\rightarrow 0} v(t, t\beta)$.

Now we turn our attention to the connection of these main results to the convergence analysis of the primal-dual infeasible 
interior-point algorithm.  Indeed, 
the pair \eqref{Pptbd} and \eqref{Dptbd} appears often in the analysis of infeasible interior-point algorithms.
In particular, primal-dual infeasible interior-point algorithms typically generate a sequence of feasible solutions
to ${\bf P}(t^k,t^k)$ and ${\bf D}(t^k,t^k)$, where $\Ip$ and $\Id$ are determined
by the initial value of the algorithm and $t^k$ is a positive sequence converging to 0. 
By Theorem~\ref{2.2} , the common optimal value $v(t^k, t^k)$ of  ${\bf P}(t^k,t^k)$ and ${\bf D}(t^k,t^k)$ converges to
$\av(\pi/4)$ which is between $v({\bf P})$ and $v({\bf D})$.
Therefore, if we can show that an infeasible interior-point algorithm generates a sequence 
which approaches $v(t^k,t^k)$ as $k\rightarrow\infty$, we can prove that that sequence converges to 
$v(\pi/4)$ in the end.
 
Exploiting this idea, we obtain the following convergence results without any assumption on
the feasibility status of the problem.  We consider  two typical well-known
polynomial-time algorithms by Zhang\cite{Zhang:1998:ESP} and Potra and Sheng \cite{ipm:Potra16}.  
But the idea can be applied to a broad class of infeasible interior-point algorithms to obtain analogous results.
They are stated formally in {Theorem \ref{5.1}} and  {Theorem \ref{5.2}}, and
summarized as follows:

{\em
\begin{enumerate}
\item The algorithms \cite{Zhang:1998:ESP,ipm:Potra16} 
generate asymptotically pd-feasible sequences with the duality gap $X^k\bullet S^k$ and $t^k$ converging to zero if and only if {\bf P} and {\bf D} are asymptotically pd-feasible.
\item If {\bf P} and {\bf D} are asymptotically pd-feasible,
the sequence of modified primal and dual objective values converges to a common value between the primal optimal value $v({\bf P})$
and the dual optimal value $v({\bf D})$ even in the presence of nonzero duality gap. 
\end{enumerate}
} 
\noindent 
The modified primal and dual objective values mentioned in the statements 
can be easily computed using the current iterate and 
do not require any extra knowledge. 

If {\bf P} and {\bf D} are not asymptotically pd-feasible, namely, if one of the problems
is strongly infeasible, the algorithms get stuck at a certain point and they fail to generate an asymptotically pd-feasible sequence
and fails to drive duality gap and $t^k$ to 0.
But the algorithms never fails to generate asymptotically pd-feasible sequences as long as the problems are asymptotically pd-feasible.

We note that Theorems \ref{5.1} and \ref{5.2} are to some extent surprising in that infeasible interior-point algorithms work 
in a meaningful manner without making any restrictive assumptions, at least in theory.
This might have interesting implications when solving SDP relaxations arising from
 hard optimization problems such as MISDP by using infeasible interior-point algorithms.   
The theorems guarantees that the modified objective function value converges to a value between the primal and
dual optimal values.  Therefore, the limiting modified objective value can  always be used to bound the 
optimal value of linear SDP relaxations obtained when solving MISDP via, say, branch-and-bound as in \cite{GPU18}.  We should mention, however, that if one tries to implement this idea, 
one would still need to find a way to overcome the severe numerical difficulties that may happen when attempting to solve singular SDPs directly. 

{Finally, while the results of this paper clarifies some aspects of the limiting behavior of infeasible interior-point algorithms when applied to a 
	problem with nonzero duality gap,
 we remark that deriving similar results for self-dual embedding approaches is still an open problem.}

\subsection{Related Work}
Our work is closely related to the perturbation theory and sensitivity analysis which are, of course, classic topics in the optimization literature.
In particular, there are a number of results on perturbation of semidefinite programs including \cite{BS00,Luo96dualityand,Luo97dualityresults,sturm_high} which were mentioned in the introduction. 
The book by Bonnans and Shapiro \cite{BS00}, for instance, has many results on the perturbation and sensitivity analysis of general conic programs that are also applicable to SDPs. See also \cite{RT74} for earlier results in the context of convex optimization. 
However, many of those results require that some sort of constraint qualification holds. 

In particular, in Chapter~4 of \cite{BS00} there is a discussion { of} a family of optimization problems having the format
\begin{equation}\label{eq:pu}
\min _{x\in X} f(x,u) \ \  \textrm{s.t.} \ \  G(x,u) \in \stdCone,
\end{equation}
where $f$ and $G$ are functions depending on the parameter $u$ and $\stdCone$ is a closed convex set in some Banach space.
Denote by $v(u)$, the optimal value of \eqref{eq:pu}.
For some fixed $u_0$, many results are proved about the continuity of $v(\cdot)$ \cite[Proposition~4.4]{BS00}, or the directional derivatives of $v(\cdot)$ in a neighborhood of $u_0$ \cite[Theorem~4.24]{BS00}. 

However, these existing results do not cover the situations we will deal in this paper. \cite[Proposition~4.4]{BS00}, for example, requires a condition called \emph{inf-compactness}, which implies, in particular, that the set of optimal solutions of the problem associated to $v(u_0)$ be compact. 
\cite[Theorem~4.24]{BS00}, on the other hand, requires that the set of optimal solutions associated to $v(u_0)$ be non-empty. In contrast, neither compactness nor non-emptiness is  assumed in this paper.

The perturbation we consider is closely related to the infeasible central path appearing in the primal-dual infeasible interior-point algorithms.  
In fact, we use some properties of the infeasible central path in our proof.  
The papers \cite{Lu:2004:EBL,PS_2004} showed the analyticity of the entire trajectory including the end 
point at the optimal set under the existence of primal-dual optimal solutions satisfying strict complementarity conditions.
A very recent paper \cite{doi:10.1137/19M1289327} analyzes the limiting behavior of singular infeasible central paths taking into 
account the singularity degree.
Therein, the authors analyze the speed of convergence under the assumption that the feasible region exists and 
is bounded.  No strong feasibility assumption is made, although we remark that if the feasible region of a primal SDP is non-empty and bounded, then its dual counterpart must satisfy Slater's condition.
While their analysis conducts a detailed limiting analysis 
on the asymptotic behavior of the central path, 
our analysis deals with the limiting behavior of the optimal value of the perturbed system under
weaker assumptions.

In reality, it may be necessary to estimate the error of an approximate optimal solution to a problem with a finite perturbation.  
In this regard, an interesting and closely related topic to the limiting perturbation analysis is error bounds.  
The error bound analysis is relatively easy under primal-dual strong feasibility, but it becomes much harder for singular SDPs.
See \cite{sturm_handbook,sturm_error_2000} for SDP and SOCP, and \cite{Lourenco17} for a more general class of convex programs.  
The relationship between forward and backward errors 
of a semidefinite feasibility system is closely related 
to  its singular degree, which, roughly, is defined as the number of
facial reduction steps necessary for regularizing the problem.
Recently, some analysis of limiting behaviors of the external (or infeasible) central path involving singularity degree is 
developed in \cite{doi:10.1137/19M1289327}.   
Finally,  we mention \cite{WS16} which conducted a sensitivity analysis of SDP under perturbation of the coefficient matrices \lq\lq$A_i$''. 

\section{Preliminaries}\label{sec:prel}

In this section, we introduce three ingredients of this paper, namely, asymptotic strong duality, infeasible interior-point algorithms
and real-algebraic geometry.

\subsection{Asymptotic Strong Duality}

A main difference between the duality theory in linear programming and general convex programming is that the latter requires some regularity conditions
for  strong duality to hold.  If such regularity condition is violated, then the primal and dual may have nonzero duality gap \cite{Ramana95anexact}.
Nevertheless, the so-called \emph{asymptotic strong duality} holds even in such singular cases
\cite{Duffin1957,bams/1183530289,BS00,Renegar_book_2001,Luo96dualityand,Luo97dualityresults,sturm_high}.
Here we quickly review the result and work on it a bit to derive a modified and simplified version suitable for our purposes.

Let $\hbox{\rm a-val}({\bf P})$ and $\hbox{\rm a-val}({\bf D})$ be
\begin{align}
\hbox{\rm a-val}({\bf P})& \equiv \lim_{\varepsilon\downarrow0} \inf_{\|\Delta b\|<\varepsilon} \inf\{C\bullet X |\  A(X) = b+\Delta b,\ X\succeq 0\}, \nonumber \\
\hbox{\rm a-val}({\bf D}) &\equiv \lim_{\varepsilon\downarrow0} \sup_{\|\Delta C\|<\varepsilon} 
\sup \{b^T y| \ C+\Delta C -\sum_{i} A_i y_i \succeq 0 \}. \label{avalD}
\end{align}
Here, $\hbox{\rm a-val}({\bf P})$ and $\hbox{\rm a-val}({\bf D})$ are called the \emph{asymptotic optimal values of {\bf P} and {\bf D}}, respectively \cite{Renegar_book_2001}. 
(It is also called {\em subvalue} in \cite{Duffin1957,bams/1183530289,BS00,Luo96dualityand,Luo97dualityresults}.)
The following asymptotic duality theorem holds, see also 
\cite[Theorem~1]{Duffin1957}, \cite[Lemmas~1 and 2]{bams/1183530289},  
\cite[Theorem~2]{Luo96dualityand},\cite[Theorem~6]{Luo97dualityresults} for similar statements.

{\bf Theorem}\ [Asymptotic Duality Theorem, e.g.,  {\cite[{{Theorem}~3.2.4}]{Renegar_book_2001}}]
{\em
\begin{enumerate}
\item If {\bf P} is asymptotically feasible, then, $\hbox{\rm a-val}({\bf P}) = v({\bf D})$.
\item If {\bf D} is asymptotically feasible, then, $\hbox{\rm a-val}({\bf D}) = v({\bf P})$.
\end{enumerate}
}

\medskip\noindent
Note that the Asymptotic Duality Theorem  includes the cases where $\hbox{\rm a-val}(\cdot)=\pm\infty$.

Now we develop a simplified version of the  Asymptotic Duality Theorem.
Let $\varepsilon \geq 0$, and let ${\bf D}$($\varepsilon$) be
${\bf D}(\varepsilon,0)$, i.e., 
the relaxed dual problem
\begin{equation}\label{relaxedual}
\max\,\, b^T y \ \ \hbox{s.t.} \ C - \sum_{i=1}^m A_i y_i +\varepsilon \Id = S,\ \ \ S\succeq 0.
\end{equation}
According to the notation introduced in Section \ref{subsec2.2},
the optimal value of \eqref{relaxedual} is written as $v(\varepsilon, 0)$.  Recall also that
\[
\lim_{\varepsilon\downarrow 0} v(\varepsilon, 0) = \av(0) =\av({\bf D}).
\]

Next we consider an analogous relaxation at the primal side.
Notice that (\ref{relaxedual}) is obtained by shifting the semidefinite cone by $-\varepsilon \Id$.
The analogous perturbation of the primal problem is given by
\begin{equation}\label{PPrelax}
\min \ C\bullet \widetilde X \ \ \hbox{s.t.}\ A(\widetilde X)=b, \ \widetilde X\succeq -\eta \Ip,
\end{equation}
where $\eta\geq0$.
Letting $X\equiv\widetilde X + \eta \Ip$, we obtain 
\begin{equation}\label{PPrelax2}
 \min \ C\bullet X -\eta C\bullet \Ip \ \ \hbox{s.t.}\ A(X) = b+ \eta A(\Ip),\ \ X\succeq 0.
\end{equation}
The optimal value of \eqref{PPrelax} is monotone decreasing in $\eta$, because the feasible region enlarges as
$\eta$ is increased (strictly speaking, it does not shrink).  Observe also that this problem is ${\bf P}(0, \eta)$ with the objective function 
shifted by a constant $- \eta C\bullet \Ip$.  Since this constant vanishes as $\eta\rightarrow 0$, we obtain 
\[
\av\left(\frac{\pi}2\right)=\av({\bf P})= \lim_{\eta\downarrow 0}v(0,\eta)=
 \lim_{\eta\downarrow 0}\{\hbox{The\ optimal\ value\ of\ (\ref{PPrelax2})}\}.
\]

%

{ Now we prove Theorem \ref{3.2}, which is a simplified version of the 
asymptotic duality theorem discussed above. 
Compared with the asymptotic duality results discussed in \cite{Duffin1957,bams/1183530289,BS00,Luo96dualityand,Luo97dualityresults, Renegar_book_2001}, 
the key difference is that we only consider perturbations along a single direction in each of the primal and dual problems, while in the aforementioned works the perturbation space is larger. Indeed, in the 
 Asymptotic Duality Theorem  (as stated above), the  perturbation space is $\|{\Delta b}\| < \epsilon$ and $\| \Delta C \| < \epsilon$ at the primal and dual sides, respectively.
 In contrast, in Theorem \ref{3.2} below,  we only consider perturbations along a single direction at each of the primal and dual problems (i.e., along $I_p$ and $I_d$, respectively). 
 Since it is not a priori obvious that the smaller perturbation space is still enough to close the duality gap, we provide a detailed proof showing how to go from the Asymptotic Duality Theorem to Theorem~\ref{3.2}.

}
\begin{theorem}\label{3.2}The following statements hold.
\begin{enumerate}
\item If {\bf D} is asymptotically feasible, then 
\begin{equation}\label{T3D}
\av(0)=\av({\bf D}){=\lim_{\varepsilon\downarrow 0}v(\varepsilon, 0)}=v({\bf P}).
\end{equation}
\item If {\bf P} is asymptotically feasible, then
\begin{equation}\label{T3P} 
\av({\pi}/2)=\av({\bf P}){=\lim_{\eta\downarrow 0}v(0,\eta)}=v({\bf D}).
\end{equation} 
\end{enumerate}

\end{theorem}

\begin{proof}
Recall that by definition (see \eqref{eq:vadvap}), we have 
$\av(0)=\av({\bf D})$ and $\av({\pi}/2)=\av({\bf P})$.

First we show that $\av({\bf D})=v({\bf P})$.
From the Asymptotic Duality Theorem, $\hbox{\rm a-val}({\bf D}) = v({\bf P})$ holds including the special cases where $\hbox{\rm a-val}({\bf D}) = \pm \infty$.
We observe that $\hbox{a-val}({\bf D})$ satisfies
\[
\hbox{\rm a-val}({\bf D})=\lim_{\varepsilon\downarrow0} \sup_{y,\Delta C}\ \{b^T y \mid C+\Delta C - \sum_i A_i y_i \succeq 0,\ \|\Delta C\| \leq \varepsilon\},
\]
where $\|\Delta C\|<\varepsilon$ in \eqref{avalD} is changed to $\|\Delta C\|\leq\varepsilon$.

Since $\av(0)$ is obtained by restricting the condition on $\Delta C$ from ``$\|\Delta C\|\leq\varepsilon$'' to 
``$\Delta C = \varepsilon \Id/\|\Id\|$'', we obtain $\av(0) \leq\hbox{\rm a-val}({\bf D})= v_a({\bf D})$.  
We also have the converse inequality $\av(0) \geq v_a({\bf D})$ because
\begin{align*}
\hbox{\rm a-val}({\bf D}) &= \lim_{\varepsilon\downarrow0} \sup \{ b^T y \mid C+\Delta C - \sum_i A_i y_i \succeq 0,\ \|\Delta C\|\leq \varepsilon \} \\
&= \lim_{\varepsilon\downarrow0} \sup \{ b^T y \mid C+\Delta C - \sum_i A_i y_i \succeq 0,\ -\varepsilon I \preceq \Delta C \preceq \varepsilon I\} \\
&\leq \lim_{\varepsilon\downarrow0} \sup \{ b^T y \mid C+\Delta C - \sum_i A_i y_i \succeq 0,\ \Delta C \preceq \varepsilon I\} \\
&\leq \lim_{\varepsilon\downarrow0} \sup \{ b^T y \mid C+\Delta C - \sum_i A_i y_i \succeq 0,\ \Delta C \preceq \varepsilon \|\Id^{-1}\|\Id\} \\
&\leq \lim_{\varepsilon\downarrow0} \sup \{b^T y \mid C+\varepsilon \Id - \sum_i A_i y_i \succeq 0 \} =\av({\bf D}).
\end{align*}
Here we used $I \preceq \|\Id^{-1}\|\Id$ for the second inequality.
{The proof of item 1 is complete.}

We proceed to prove item 2.
{From the Asymptotic Duality Theorem again, we have $v({\bf D})=\hbox{a-val}({\bf P})$.
Hence, for the sake of proving assertion~2, it suffices to show that $\av({\bf P})=\hbox{a-val}({\bf P})$.
}  
The proof of the inequality $\av({\bf P})\geq\hbox{a-val}({\bf P})$ is analogous to the proof for $\av({\bf D})\leq\hbox{a-val}({\bf D})$.  
We will now show the converse inequality. 
If $\hbox{a-val}({\bf P})=+\infty$, 
then $\av({\bf P})\geq\hbox{a-val}({\bf P})$ implies that $\av({\bf P}) = +\infty$.
Therefore, in what follows we assume that $\hbox{a-val}({\bf P})<+\infty$.

{By assumption, ${\bf P}$ is not strongly infeasible (see Section~\ref{sec:setup}).}
By the definition of $\hbox{a-val}({\bf P})$, for every $\varepsilon > 0$ sufficiently small, there exist ${X_{\varepsilon}}$ and
${\Delta b_{\varepsilon}}$ such that $\|\Delta b_{\varepsilon}\|\le \varepsilon$, ${X_{\varepsilon}}$ is feasible to ``$A(X)=b + {\Delta b_{\varepsilon}}, \ X\succeq 0$'', and 
\begin{equation}
\hbox{a-val}({\bf P}) = \lim _{\varepsilon \downarrow 0 } C \bullet X_{\varepsilon}.  \label{eq:xhate}
\end{equation}
Note that this is still valid even when $\hbox{a-val}({\bf P})=-\infty$.

In addition, the fact that ${\bf P}$ is not strongly infeasible implies the 
existence of a solution to the system ``$A(X') = b$''. 
As a consequence, ``$A(Y)=\Delta b_{\varepsilon}$'' too has a solution when $\Delta b_{\varepsilon}$ is as described above.
{Otherwise, ``$A(X) = b+\Delta b_{\varepsilon}$'' is infeasible, contradicting the existence of $X_{\varepsilon}$ above.}

{ Next, we show that
there exists $M > 0$ depending only on  $A$ such that ``$A(Y)=\Delta b_{\varepsilon}$'' has a solution with norm bounded by $M \norm{\Delta b_{\varepsilon}}$.
{Let $\mathcal{V}$ denote the set of solutions to ``$A(Y)=\Delta b$'' and let $S$ be a symmetric matrix. Denote by $\text{dist}\,(S,\mathcal{V}) $ the Euclidean distance between $S$ and $\mathcal{V}$. Hoffman's lemma (e.g., \cite[Theorem~11.26]{Gu10}) says that there exists a constant $M$ depending on $A$ but not on $\Delta b$ such that for every $S$, we have that $\text{dist}\,(S,\mathcal{V}) $ is bounded above by $M\norm{\Delta b-A(S)}$. Taking  $S = 0$, we conclude the existence of $Y$ satisfying $A(Y)= \Delta b$ and 
  $\norm{Y} \leq M\norm{\Delta b}$. }}
  
Let $Y_{\varepsilon}$ be one such solution. Then $\|Y_{\varepsilon}\|\le M\|\Delta b_{\varepsilon}\|\le M\varepsilon$ for each sufficiently small $\varepsilon>0$ and hence
\begin{equation}
  \lim_{\varepsilon\downarrow 0}\|Y_{\varepsilon}\|=0. \label{eq:normYe}
\end{equation}
Observing that $\|\Ip^{-1}\|\Ip \succeq I$ and 
$\|Y_{\varepsilon}\|I - Y_{\varepsilon}\succeq 0$ yield $\|Y_{\varepsilon}\|\|\Ip^{-1}\|\Ip - Y_{\varepsilon}\succeq 0$,
  we let 
\begin{equation*}
X'_{\varepsilon} \equiv X_{\varepsilon} + \|Y_{\varepsilon}\|\|\Ip^{-1}\|\Ip - Y_{\varepsilon}. \label{eq:Xprime}
\end{equation*}
With that, $X'_{\varepsilon}$ is  positive semidefinite and is a feasible solution to 
${\bf P}(0, \eta)$ with $\eta=\|Y_{\varepsilon}\|\|\Ip^{-1}\|$ (see (\ref{Pptbdalt})).  Furthermore,
\begin{equation}\label{eq:limit}
|C\bullet X_{\varepsilon} - C\bullet X'_{\varepsilon}|
=|C\bullet(\|Y_{\varepsilon}\|\|\Ip^{-1}\|\Ip - Y_{\varepsilon})| \leq 2\|C\| \|Y_{\varepsilon}\|\|\Ip^{-1}\|\| \Ip\|_F,
\end{equation}
which approaches $0$ by driving $\varepsilon\to 0$ because of \eqref{eq:normYe}.

We are now ready to show the desired assertion.
Notice that
we have $\lim_{\varepsilon\downarrow 0} C \bullet X'_{\varepsilon}\ge \av({\bf P})$, since $X'_{\varepsilon}$ is feasible to ${\bf P}(0, \|Y_{\varepsilon}\|\|\Ip^{-1}\|)$ and \eqref{eq:normYe} holds.
This fact combined with \eqref{eq:xhate} and \eqref{eq:limit} implies $\av({\bf P})\leq\hbox{a-val}({\bf P})$.
The proof is complete.
%
\qed
\end{proof}


Theorem~\ref{3.2} motivates our subsequent discussion and leads naturally to an examination of what happens when {\bf P} and {\bf D} are \emph{simultaneously} perturbed, which is the focus of Theorems~\ref{2.1}, \ref{2.2} and \ref{4.2}. 

\subsection{Infeasible Primal-dual Interior-point Algorithms}\label{IPM}

We introduce some basic concepts of infeasible primal-dual interior-point algorithms for SDP
\cite{Zhang:1998:ESP,ipm:Potra16,YFK2003,SDPT3_2003}. 
This is because our analysis leads to a novel convergence property of the infeasible primal-dual 
interior-point algorithms when applied to singular problems.  We also need some
theoretical results about infeasible interior-point algorithms 
in the proof of Theorem \ref{2.1}.  In this subsection, we assume that $A_i$ $(i=1, \ldots, m)$ are
linearly independent.  
This assumption is not essential but to ensure uniqueness of $y$ and $\Delta y$  in the system of equations
of the form $S=\sum_i A_i y_i +C'$ and $\Delta S = \sum_i A_i \Delta y + R'$ with respect to $(S,y)$ and $(\Delta S,\Delta y)$, respectively, where $C'$ and $R'$ are constants, which appear throughout the analysis.

\subsubsection{Outline of infeasible primal-dual interior-point algorithms}\label{3.2.1}

Primal-dual interior-point methods for {\bf P} and {\bf D} are based on the following optimality conditions:
\begin{equation}\label{pd-formulation}
XS = 0,\ \ \ C - \sum_i A_i y_i = S,\ \ \ A(X) =b\ \ \ X \succeq 0,\ \ \ S \succeq 0.
\end{equation}
Rather than solving this system directly, a relaxed problem
\begin{equation}\label{pd-r-formulation}
XS = \nu I,\ \ \ C - \sum_i A_i y_i = S,\ \ \ A(X) = b, \ \ X \succeq 0,\ \ \ S \succeq 0,
\end{equation}
is considered, 
where $\nu > 0$.  The algorithm solves (\ref{pd-formulation}) by solving (\ref{pd-r-formulation}) approximately 
and reducing $\nu$ gradually to zero repeatedly.   This amounts to following the central path
\begin{equation}\label{cpath}
\{(X_\nu,S_\nu,y_\nu) \mid  (X,S,y)=(X_\nu,S_\nu,y_\nu)\hbox{\ is\ a\ solution to \eqref{pd-r-formulation}},\ \nu\in (0, \infty]\}
\end{equation}
towards ``$\nu = 0$''.  
Let us take a closer look at the algorithm proposed by Zhang, more precisely, Algorithm-B of  \cite{Zhang:1998:ESP}.

Let $(X, S, y)$ be the current iterate such that $X\succ 0$ and $S\succ 0$.  
The method employs the Newton direction to solve the system (\ref{pd-r-formulation}).
More precisely, 
the first equation $XS=\nu I$ is replaced with an equivalent symmetric reformulation
\begin{equation} \label{Phi}
\Phi(X,S) = \frac12(PXSP^{-1}+P^{-1}SXP)=\nu I,
\end{equation}
where $P$ is a constant nonsingular matrix. In Zhang's algorithm, the constant matrix $P$ is set to $S^{1/2}$.
Then we consider a modified nonlinear system of equations to \eqref{pd-r-formulation} where $XS=\nu I$ is replaced with \eqref{Phi}.
The Newton direction 
$(\Delta X, \Delta S, \Delta y)$ for that modified system
at the point $(X, S, y)$ 
is the  unique solution to the following system of linear equations.
\begin{equation}\label{Newton}
\Phi(X,S)+ L_\Phi(\Delta X, \Delta S) = \nu I, \ C - \sum_i A_i (y_i +\Delta y_i) = S + \Delta S,\ A(X+\Delta X) = b,
\end{equation}
where $L_\Phi$ is a linearization of $\Phi(X,S)$. 

Starting from the $k$th iterate $(X^k,S^k,y^k)=(X,S,y)$, the next iterate $(X^{k+1}, S^{k+1}, y^{k+1})$ is determined as:
\begin{equation}\label{catA}
(X^{k+1}, S^{k+1}, y^{k+1}) = (X^k, S^k, y^k) + s^k (\Delta X, \Delta S, \Delta y).
\end{equation}
The stepsize $0 < s^k \leq 1$ is chosen not only so that $X^{k+1}$ and $S^{k+1}$ are strictly positive but also carefully 
so that they stay close to the central path  in order to ensure good convergence properties.  Then $\nu$ is updated appropriately and the iteration continues.  

Now we briefly describe another representative polynomial-time infeasible primal-dual interior-point algorithm developed by Potra and Sheng \cite{ipm:Potra16}. Let $(X^0, S^0, y^0)$ be a point satisfying $X^0\succ 0$ and $S^0\succ 0$ and consider the path defined
as follows.
\begin{eqnarray}
&&\{(X,S, y) \mid  XS = t I,\ \ \ C - \sum A_i y_i -S = t(C - \sum A_i y_i^0 -S^0),\nonumber \\
&&\ \ \ A(X) - b =t(A(X^0)-b), \ X \succeq 0,\ S \succeq 0,\ t\in (0,1]\}. \label{cpath2}
\end{eqnarray}
The algorithm follows this path by driving $t\rightarrow 0$ and using a predictor-corrector method.

We note that polynomial-time convergence is proved for both algorithms \cite{Zhang:1998:ESP,ipm:Potra16} assuming
the existence of optimal solutions $(X^*,S^*,y^*)$ to {\bf P} and {\bf D}. In the analysis,  the initial iterate 
$(X^0,S^0, y^0)$ is set to $(\rho_0 I, \rho_1 I, 0)$ where $\rho_0$ and $\rho_1$ are selected to be 
large enough in order to satisfy the conditions $X^0-X^*\succ 0$ and $S^0-S^*\succ 0$.  
Although the polynomial convergence analysis was conducted using this initial iterate,  
the algorithms themselves can be applied to any SDP problem by choosing
$(X^0, S^0, y^0)$ such that $X^0\succ 0$ and $S^0\succ 0$ as the initial iterate.   

In many practical implementations of the algorithm \cite{YFK2003,SDPT3_2003}, they take different 
stepsizes in the primal and dual space for the sake of practical efficiency.
For simplicity of presentation, we only analyze the case \eqref{catA} which corresponds to the situation where we take the same stepsize in the primal-dual space.  

The following well-known property connects Theorems \ref{2.1} and \ref{2.2} to the analysis of infeasible interior-point algorithms.
\begin{proposition}\label{StpIPM}

Let $X^0 \succ 0$ and $S^0 \succ 0$, and let  $\{(X^k, S^k, y^k)\}$ be a sequence
generated by the primal-dual infeasible interior-point algorithms in \cite{Zhang:1998:ESP,ipm:Potra16}
with initial iterate $(X^0, S^0, y^0)$.
Let $I'_d\equiv S^0 - (C-\sum A_i y_i^0)$ and let $I'_p \equiv X^0 - \widetilde X$ where $A(\widetilde X) = b$. Then, there exists a nonnegative 
sequence $\{t^k\}$ such that the following equations hold:
\begin{equation} \label{0223}
(C +  t^k I'_d) - \sum_i A_i y^k = S^k, \quad A(X^k) = b +  t^k A(I'_p).
\end{equation}
{\rm (}cf. The linear equality constraints of \eqref{Pptbd} and \eqref{Dptbd}{\rm )}
\end{proposition}
\begin{proof}
This  result is a fundamental tool used in the analysis of the algorithms in \cite{Zhang:1998:ESP,ipm:Potra16}.
For the sake of completeness, here we prove the result only for Zhang's algorithm.

We prove the first relation of (\ref{0223}) by induction.  For $k=0$, the proposition holds by taking $t^0\equiv 1$.
Suppose that the relation (\ref{0223}) holds for $k$, then, 
the search direction $(\Delta X, \Delta S, \Delta y)$ is the solution to the linear
system of equations (\ref{Newton}) with $(X, S, y) = (X^k, S^k, y^k)$.
Because of the second equation of (\ref{Newton}), we have
\[
C- \sum A_i (y_i^k +\Delta y_i) - (S^k+\Delta S)=0.
\]
Therefore, 
\[
C- \sum A_i (y_i^k + s^k\Delta y_i) -(S^k + s^k\Delta S) = (1-s^k)(C- \sum A_i y_i^k - S^k).
\]
Since $y_i^{k+1} = y_i^k + s^k\Delta y_i, S^{k+1} = S^k + s^k\Delta S$ and $t^{k+1} = (1-s^k) t^k$, we obtain
\[
C - \sum A_i y_i^{k+1} - S^{k+1} = (1-s^k)t^k I'_d=t^{k+1} I'_d
\]
as we desired, 
because $C- \sum A_i y_i^k - S^k = t^k(C -\sum A_i y_i^0 - S^0)=t^k I'_d$ holds by the induction assumption.
The primal relation, i.e., the right side in \eqref{0223}, follows similarly.
\qed\end{proof}
{\bf Remark} 
In view of Proposition~\ref{StpIPM}, by convention,
we treat $t^k$ as  a part of iterates of the algorithms.  By its construction, we have $t^0=1$ and 
\begin{equation}  \label{recur_t}
t^{k+1}=\prod_{l=0}^k (1-s^l)
\end{equation}
for $k=0, 1, \ldots$

\subsubsection{Path formed by points on the central path of perturbed problems}\label{3.2.2}
We fix $\nu$ to be a positive number, and consider the following system of equations and semidefinite
conditions parametrized by $t > 0$:
\begin{equation}\label{path}
\begin{aligned}
&&XS-\nu I=0,\ \ \  C + t\alpha \Id - \sum_i A_i y_i- S =0 , \\
&&A(X - t \beta \Ip) -b = 0,
\\ 
&&X\succeq 0,\ \ S \succeq 0.
\end{aligned}
\end{equation}
We denote by $w_\nu(t) \equiv (X_\nu(t), S_\nu(t), y_\nu(t))$ the solution of \eqref{path} (if it exists). 
 If the problem is asymptotically pd-feasible, for any $t>0$,
 { ${\bf P}(t\alpha,t\beta)$ and ${\bf D}(t\alpha,t\beta)$ are strongly feasible.
 Then} the solution of \eqref{path}
 defines a point on the central path {with parameter $\nu$} of the primal-dual pair of strongly feasible SDP:
\begin{equation}\label{ptbP}
\min\ (C + t\alpha \Id) \bullet X\ \ \ \hbox{s.t.} \  
A(X - t \beta \Ip)  = b,\ \ \ X\succeq 0
\end{equation}
and 
\begin{equation}\label{ptbD}
\max\ \sum_i (b_i + t\beta A_i\bullet \Ip) y_i\ \ \ \hbox{s.t.}\ C + t\alpha \Id - \sum_i A_i y_i= S, \ \ \ S \succeq 0,
\end{equation}
where we note that $t$ is fixed in \eqref{ptbP} and \eqref{ptbD}.
In this case, $w_\nu(t)$ is ensured to exist and is uniquely determined for all $t \in (0,\infty)$ (due to the assumption of 
linear independence of $A_i$, $i=1, \ldots, m$).
Moreover, the set 
\begin{equation} \label{pathC}
{\cal C}\equiv\{w_{\nu}(t) \mid t\in (0,\infty)\}
\end{equation}
forms an analytic path running through $\iPSDcone{n}\times\iPSDcone{n}\times\Re^m$.  
The existence and analyticity of ${\cal C}$ is a folklore result (e.g., \cite{Lu:2004:EBL,PS_2004}), but  we outline a proof in the Appendix~\ref{app} based on 
a result in \cite{Monteiro:1999:PCN}. {  We note that {the existence and analyticity of the path just relies on
	local conditions, so, the existence of optimal solutions of {\bf P} and {\bf D} is not necessary. 
}}
A special case where $\nu=1$ and $C=0$ is analyzed in \cite{sremac2017complete} in the context of
facial reduction.

Since $A(X_\nu(t))=b+t\beta A(\Ip)$, $C + t\alpha \Id - \sum_i A_i y_{\nu i}(t) = S_\nu(t)$, and
$X_\nu(t) S_\nu(t)=\nu I$ hold, we have
\begin{eqnarray}
0&\leq& (C+t\alpha \Id)\bullet X_\nu(t) -\sum_{i=1}^m (b_i+t\beta A_i\bullet  \Ip) y_{\nu i}(t)\nonumber\\
&=&(C+t\alpha \Id)\bullet X_\nu(t) -\sum_{i=1}^m A_i\bullet X_\nu(t) y_{\nu i}(t)  \nonumber \\
\nonumber\\
&=& S_\nu(t)\bullet X_\nu(t) = {\rm Tr}(X_\nu(t)S_\nu(t))={\rm Tr}(\nu I) = n \nu. \label{223a}
\end{eqnarray} 

\def\vopt{v_{\rm opt}}
Let us denote by $\vopt(t)$ the common optimal value of (\ref{ptbP}) and (\ref{ptbD}).
Since $\vopt(t)$ is between  
$(C+t\alpha \Id)\bullet X_\nu(t)$ and $\sum_{i=1}^m (b_i+t\beta A_i\bullet  \Ip) y_{\nu i}(t)$, i.e.,
\begin{equation}\label{223b}
\vopt(t)\in\left [\sum_{i=1}^m (b_i+t\beta A_i\bullet  \Ip) y_{\nu i}(t), (C+t\alpha \Id)\bullet X_\nu(t)\right]
\end{equation}
holds by weak duality, we see, together with \eqref{223a}, that
\begin{equation}\label{optvbd}
0\le (C+t\alpha \Id)\bullet X_\nu(t) - \vopt(t) \leq n\nu
\end{equation}
holds for each $t > 0$.

\subsection{Semialgebraic sets and the Tarski-Seidenberg Theorem}
A set $S$ in $\Re^k$ is called \emph{basic semialgebraic} if it can be written as the set of solutions of finitely many polynomial equalities and strict polynomial inequalities. Then, a 
set is said to be \emph{semialgebraic} if it is a union of finitely many basic semialgebraic sets.
 In particular, a semialgebraic set in $\Re$ is a union of finitely many points and intervals.
For $x = (x_1, \ldots, x_k) \in \Re^k$,  let $T(x)$ be a coordinate projection to $\Re^{k-1}$ defined as $T(x)\equiv(x_2,\ldots,x_n)$.
The Tarski-Seidenberg Theorem states that a coordinate projection of a semialgebraic set is again a semialgebraic set
in the lower-dimensional space, and described as follows.

\ \\
\noindent
{\bf  Tarski-Seidenberg Theorem}\ (e.g. Theorem 2.2.1 of \cite{bochnak_coste_roy}) \\
{\it 
Let $W \subseteq \Re^k$ be a semialgebraic set.  Then, $T(W)$ is a semialgebraic set in $\Re^{k-1}$.
}

\section{Proof of the Main Results}\label{sec:proof}


%
%
%



In this section, we prove Theorems \ref{2.1} and \ref{2.2}.  
We start with some basic properties of $v(\varepsilon,\eta)$.
\begin{proposition}\label{31} If the problem is asymptotically pd-feasible, the following statements hold. 
\begin{enumerate}
\item $v(\varepsilon, \eta)$ is well-defined for all $(\varepsilon, \eta) \geq 0$ not equal to $(0, 0)$.  Furthermore,
\begin{enumerate}[$(i)$]
\item $\lim_{\varepsilon\downarrow 0} v(\varepsilon, 0) = v({\bf P})$ and
\item $\lim_{\eta\downarrow 0} v(0,\eta) = v({\bf D})$
\end{enumerate}
hold including the cases where their values are $\pm \infty$.
\item $v(\varepsilon,\eta)$ is a
 monotone increasing concave function in $\varepsilon$. 
 (From item~$1.$,  if $\eta > 0$, $v(\varepsilon,\eta)$ is   well-defined over $[0,\infty)$. 
 If $\eta = 0$, $v(\varepsilon,\eta)$ is well-defined over $(0,\infty)$.)

\item $v_P(\varepsilon,\eta)\equiv v(\varepsilon,\eta)-\eta C\bullet \Ip-\eta\varepsilon \Id \bullet \Ip$ is 
{a monotone decreasing and convex} function in $\eta$.  (From item~$1.$,  if $\varepsilon > 0$, $v_P(\varepsilon,\eta)$ is   well-defined over $[0,\infty)$. 
If $\varepsilon = 0$, $v_P(\varepsilon,\eta)$ is well-defined over $(0,\infty)$.)
\end{enumerate}
\end{proposition}
\begin{proof}
Item~1.~follows directly from Theorem \ref{3.2}.    
Next, we move on to item~2.
Let $(\varepsilon_1,\eta)\geq 0$, $(\varepsilon_2, \eta) \geq 0$ and,
without loss of generality, we may assume that $0 \leq \varepsilon_1 < \varepsilon_2$. 
By definition, $v(\varepsilon,\eta)$ coincides with 
$v({\bf D}(\varepsilon, \eta))$ whenever $v(\varepsilon,\eta)$ is well-defined, see Section~\ref{subsec2.2}.
Then, $v(\varepsilon,\eta)$ is monotonically increasing in  $\varepsilon$ because if $y$ is feasible for ${\bf D}(\varepsilon_1, \eta)$ then $y$ is feasible for ${\bf D}(\varepsilon_2, \eta)$ too.
Next, we prove concavity and we will start by first considering the case $\eta > 0$. 

There are two sub-cases to consider: when $\varepsilon_1=0$ and when $\varepsilon_1>0$.  In the latter sub-case, ${\bf D}(\varepsilon_1, \eta)$ and
${\bf D}(\varepsilon_2, \eta)$ are both feasible, since $\varepsilon_1, \varepsilon_2$ and $\eta$ are all positive and  asymptotic primal-dual feasibility was assumed.
For simplicity, we define $\hat b$ as the vector corresponding to the objective function of ${\bf D}(\varepsilon, \eta)$ so that
\[
\hat b^Ty = \sum_{i=1}^m ({b}_i + \eta A_i\bullet \Ip)y_i,\qquad \forall y \in \Re^m.
\] 
We let $y^k$ and $\bar{y}^k$ be sequences of feasible solutions of ${\bf D}(\varepsilon_1, \eta)$ and 
${\bf D}(\varepsilon_2, \eta)$
satisfying 
\begin{equation*}
\hat b^Ty^k \to v({\bf D}(\varepsilon_1, \eta)) \quad \text{and}\quad
\hat b^T \bar{y}^k \to v({\bf D}(\varepsilon_2, \eta)).
\end{equation*}
Then, for $t \in [0,1]$, we have that $t y^k+(1-t) \bar{y}^k $ is a feasible solution to ${\bf D}(t\varepsilon_1+(1-t)\varepsilon_2, \eta)$
with objective value $\hat{b}^T(ty^k+(1-t)\bar{y}^k)$.  Then it follows 
\begin{align*}
 v({\bf D}(t\varepsilon_1+(1-t)\varepsilon_2, \eta))& = v(t\varepsilon_1+(1-t)\varepsilon_2, \eta)\\
    &\geq \hat{b}^T(ty^k+(1-t)\bar{y}^k) = t \hat{b}^Ty^k+ (1-t) \hat{b}^T \bar{y}^k.
\end{align*}
Taking the limit with respect to $k$, we obtain
\begin{equation}\label{conv0302}
 v(t\varepsilon_1+(1-t)\varepsilon_2,\eta) \geq t v(\varepsilon_1, \eta)+(1-t) v(\varepsilon_2, \eta)
\end{equation}
as we desired. 

Now we deal with the sub-case where $\varepsilon_1=0$.  
By assumption, we have $\varepsilon_2>\varepsilon_1=0$, implying that 
$v(\varepsilon_2, \eta)$ is finite.
Then, we can proceed analogously except 
that ${\bf D}(0, \eta)$ may be infeasible 
so that $v(\varepsilon_1,\eta) =v(0, \eta)=-\infty$.  
However, in that case, since ${v(t\varepsilon_1+(1-t)\varepsilon_2,\eta) =v((1-t)\varepsilon_2, \eta)}$ is finite
for all $t\in [0,1)$, we see that \eqref{conv0302} indeed holds. This concludes the proof for the case where $\eta > 0$.

Finally, we deal with the case $\eta = 0$. In this case, we may assume that $\varepsilon_1$ is positive, since $v(\varepsilon_1, 0)$ might not be well-defined otherwise.
By assumption, $\bf{D}$ is asymptotically feasible, so  ${\bf D}(\varepsilon, 0)$ is always feasible for $\varepsilon > 0$.
Thus the optimal value of ${\bf D}(\varepsilon_1, 0)$ is either finite or is $+\infty$.

There are two sub-cases to consider.
First, suppose that the optimal value of  ${\bf D}(\varepsilon, 0)$ is $+\infty$ for some $\varepsilon > 0$. Then, ${\bf P}(\varepsilon, 0)$ is infeasible.
However, the feasible region of ${\bf P}(\varepsilon, 0)$ is the same for all $\varepsilon > 0$, which implies infeasibility of ${\bf P}(\varepsilon, 0)$  for all $\varepsilon > 0$.
Consequently, $v(\varepsilon,0) = +\infty$ for all $\varepsilon > 0$ and \eqref{conv0302} holds.

The next sub-case is when the optimal value of  ${\bf D}(\varepsilon, 0)$  is finite for all $\varepsilon > 0$. 
In particular, $v({\bf D}(\varepsilon_1, \eta))$ and $v({\bf D}(\varepsilon_2, \eta))$ are both finite and we can proceed as in the proof of the case $\eta > 0$. This concludes the proof of item~2.

Now we prove item 3.
First, we recall that the optimal value of (\ref{PPrelax2}) (or, equivalently, (\ref{PPrelax})) is monotone decreasing in  $\eta$.
If we replace $C$ with $C +\varepsilon \Id$ in (\ref{PPrelax2}) we obtain
\begin{equation}\label{PPrelax3}
 \min \ (C+\varepsilon \Id)\bullet X -\eta (C+\varepsilon \Id)\bullet \Ip \ \ \hbox{s.t.}\ A(X) = b+ \eta A(\Ip),\ X\succeq 0.
\end{equation}
Similarly, the optimal value of \eqref{PPrelax3} is monotone decreasing in $\eta$, when $\varepsilon$ is fixed. 
Since \eqref{PPrelax3} differs from ${\bf P}(\varepsilon,\eta)$ by  
the term $(\eta(C\bullet \Ip)+\eta\varepsilon \Id\bullet \Ip)$ in the objective 
function, the optimal value of \eqref{PPrelax3} can be written as
\[
v(\varepsilon,\eta)-(\eta(C\bullet \Ip)+\eta\varepsilon \Id\bullet \Ip),
\]
which is precisely $v_P(\varepsilon,\eta)$.  Therefore   
$v_P(\varepsilon,\eta)$ is monotone decreasing with respect to $\eta$.

Finally, for fixed $\varepsilon$,  
$v_P(\varepsilon,\eta)$ and $v(\varepsilon,\eta)$ differ by a linear function in $\eta$. So 
to prove that $v_P(\varepsilon,\eta)$ is convex as a function of $\eta$, it is enough to prove that $v(\varepsilon,\eta)$ is convex as a function of $\eta$. 
This can be done analogously to the proof of item~2., so we omit the details. 
\qed\end{proof}

\medskip

In the following, we prove Theorem \ref{2.1}.   
The theorem claims that, even though $v(0,0)$ is not well-defined, the limiting value exists when approaching
$(0,0)$ along a straight line emanating from the origin to any direction of the first orthant.   


\medskip\noindent
{\bf (Proof of Theorem \ref{2.1})}

Although the result holds even if the $A_i$'s are linearly dependent, 
for simplicity sake, in this proof we assume linear independence of the $A_i$ $(i=1, \ldots, m)$. 
In addition, we write $v(t\alpha, t\beta)$ as $\vopt(t)$, since $v(t\alpha,t\beta)$ is the common optimal value to the
primal-dual pair ${\bf P}(t\alpha,t\beta)$ and ${\bf D}(t\alpha,t\beta)$.
We also assume that $\alpha > 0$ and $\beta>0$, since the proof for the case where either of $\alpha$ and $\beta$ is 0 
(but $(\alpha,\beta)\not=0$) has already been established in Proposition~\ref{31}.

Recall that we introduced the analytic path $\cal C$ in Section \ref{3.2.2} (See \eqref{path}--\eqref{pathC}). 
We follow the same notation described therein. The path $\cal C$ is parametrized by $t$. 
We divide the proof into the following two steps: 

\medskip\noindent
(Step 1) For any fixed $\nu>0$, we prove the monotonicity of $(C+t\alpha I_d)\bullet X_\nu(t)$ when $t>0$ is sufficiently  small.

\noindent
(Step 2) Prove the existence of $\lim_{t\downarrow 0} \vopt(t)$.

\medskip\noindent 
{\bf (Step 1)}

We analyze the behavior of $(C+t\alpha I_d)\bullet X_\nu(t)$ along the path ${\cal C}$ as $t\rightarrow 0$.
Recall that (\ref{path}) is the system parametrized by $t$ which defines the path ${\cal C}$.
By differentiating the three equations in (\ref{path}) with respect to $t$, we see that the following system of equations in $(t, X ,S , y, \delta X, \delta S, \delta y)$ (with semidefinite constraints on $X$ and $S$)
\begin{equation} 
\begin{array}{l}
X \delta S + \delta X S = 0, \\
\alpha \Id -\sum_i A_i \delta{y}_i = \delta S, \\
A_i\bullet (\delta X - \beta \Ip) = 0, \qquad (i=1, \ldots, m),\\
XS=\nu I, \\
 C + t\alpha I - \sum_i A_i y_i = S, \\
A_i\bullet (X - t \beta \Ip)  = b_i,\qquad (i=1, \ldots, m),\\
X\succeq 0,\ S\succeq 0, t > 0, 
\end{array}\label{alg}
\end{equation}
has a unique solution
\[
(t, X, S, y, \delta X, \delta S, \delta y) = \left(t, X_\nu(t), S_\nu(t), y_\nu(t), \frac{d X_\nu(t)}{dt},\frac{d S_\nu(t)}{dt}, \frac{d y_\nu(t)}{dt}\right)
\]
for each $t\in (0,\infty)$.
That is, \eqref{alg} is a system of equations with semidefinite constraints which 
determines the curve $(X_\nu(t), S_\nu(t), y_\nu(t))$ and its tangent 
$\left(\frac{d X_\nu(t)}{dt},\frac{d S_\nu(t)}{dt}, \frac{d y_\nu(t)}{dt}\right)$.
{
The reason for that is as follows.
{By the discussion in Section~\ref{3.2.2}, for fixed $t > 0$, $X_\nu(t),S_\nu(t)$ are uniquely defined. Since the $A_i$ are linearly independent, $y_{\nu(t)}$ must be unique  as well. In order to see that $\delta X, \delta S, \delta y$ are also uniquely determined, we take a look at the first three equations of \eqref{alg} for fixed positive definite matrices $X$ and $S$. They become linear equations in $\delta X, \delta S, \delta y$ and determine a unique solution if and only if the kernel of $\phi: (U, V, z )\mapsto (X U + V S,V + \sum _{i}A_i z_i, A(U) )$ is trivial. Suppose $\phi(U,V,z) = 0$.  Then, $U\bullet V  = 0$. Considering the first component of $\phi$, we have the equation $XU = -VS$, which implies that $\nu U = -SVS$. Taking the inner product with 
$V$, we obtain $0 = (SVS)\bullet V = \|S^{1/2}VS^{1/2}\|_F^2 $. Therefore, $S^{1/2}VS^{1/2} = 0$ and since $S$ is invertible, $V = 0$. By $\nu U = -SVS$, we have $U = 0$.	
}

Now we are ready to proceed.}  Let us denote by ${\cal D}$ the set of solutions to (\ref{alg}) as follows:
\[
{\cal D}=\{(t, X, S, y, \delta X, \delta S, \delta y)\mid (t, X, S, y, \delta X, \delta S, \delta y)\hbox{\ satisfies\ (\ref{alg}).}\}
\] 
Each element of ${\cal D}$ can be seen as 
a pair consisting  of a point on ${\cal C}$ and
its tangent.
Since the semidefinite conditions $S\succeq 0$ and $X\succeq 0$ can be written as the solution set of finitely many polynomial inequalities, ${\cal D}$ is a semialgebraic set.


Now we
claim that $(C + t\alpha \Id) \bullet X_\nu(t)$ is either monotonically increasing or monotonically decreasing for sufficiently small $t$.
To this end, we analyze the set of local minimum points and local maximum points of $(C + t\alpha \Id)\bullet X_\nu(t)$ over 
$(0,\infty)$. A necessary condition for local minimum and maximum points is:
\[
\frac{d(C + t\alpha \Id)\bullet X_\nu(t)}{dt} = (C+\alpha \Id) \bullet \frac{d X_\nu(t)}{dt} + \alpha \Id \bullet X_\nu(t) =0.
\]
Recall that for $\hat t>0$, 
$(\frac{d X_\nu}{dt}(\hat t),\frac{d S_\nu}{dt}(\hat t), \frac{d y_\nu}{dt}(\hat t))$
is the tangent part $(\delta X, \delta S, \delta y)$ of the unique solution to (\ref{alg}) with $t=\hat t$. 
With that in mind, a necessary condition for $(C+t \alpha \Id)\bullet X_\nu(t)$ to have an extreme value at $t$ is that $t$ is in the set
\[
{\cal T}_1 \equiv \{t \mid (t, X, S, y, \delta X, \delta S, \delta y)\in {\cal T}\},
\]
where
\[
{\cal T} \equiv\{ (t, X, S, y, \delta X, \delta S, \delta y)\in {\cal D} \mid \  (C+t\alpha \Id)\bullet \delta X+\alpha \Id\bullet X = 0\}.
\]
Since ${\cal D}$ is a semialgebraic set, so is ${\cal T}$.  Since ${\cal T}_1$ is the projection of ${\cal T}$ onto the $t$ coordinate, 
by applying the Tarski-Seidenberg Theorem, we see that ${\cal T}_1$ is a semialgebraic set.

Thus, ${\cal T}_1$ is a semialgebraic set contained in $\mathbb{R}$, therefore ${\cal T}_1$ can be expressed as a union of finitely many points and intervals over $\Re$.
Since $(C+t\alpha \Id)\bullet X_\nu(t)$ is an analytic function (see Section~\ref{3.2.2}), the same is true for its derivatives. 
Therefore,
if ${\cal T}_1$ contains an interval,
then the derivative of $(C+t\alpha \Id)\bullet X_\nu(t)$  with respect to $t$ must, in fact, be zero throughout $(0,\infty)$\footnote{Here we are using the fact that the zero function is analytic and if two real analytic functions $f:(0,\infty)\to \mathbb{R}$, $g:(0,\infty)\to \mathbb{R}$ coincide in some interval $(a,b)$ with $a < b$, then $f$ and $g$ coincide throughout $(0,\infty)$, e.g., \cite[Corollary~1.2.6]{KP02}.}. In particular, $(C+ t\alpha \Id)\bullet X_\nu(t)$ is constant
for all $t > 0$.  Thus, $(C+t\alpha \Id)\bullet X_\nu(t)$ is a monotonically increasing/decreasing function in this case.

Now we deal with the case where ${\cal T}_1$ consists of a finite number of points only. We recall that
$(C+t\alpha \Id)\bullet X_\nu(t)$ takes an extreme value at $t$ only if  $t\in {\cal T}_1$.
This implies that the number of { extremal} points of $(C+t\alpha \Id)\bullet X_\nu(t)$ is finite and hence $(C+t\alpha \Id)\bullet X_\nu(t)$ is monotonically increasing or monotonically 
decreasing for sufficiently small $t$.

\medskip\noindent
{\bf (Step 2)}

It follows from Step 1 that there are three possibilities.
\begin{enumerate}[(i)]
	\item $\lim_{t\downarrow 0} (C+t\alpha \Id)\bullet X_\nu(t) = \infty$,
	\item $\lim_{t\downarrow 0}(C+t\alpha \Id) \bullet X_\nu(t)= -\infty$,
	\item $\lim_{t\downarrow 0} (C+t\alpha \Id)\bullet X_\nu(t)$ is a finite value.
\end{enumerate}
First we consider cases (i) and (ii).
Recalling \eqref{optvbd}, we have
$|(C+t\alpha \Id)\bullet X_\nu(t)-\vopt(t)|\leq n\nu$.
Therefore,  $\vopt(t)$ diverges to $+\infty$ and $-\infty$, respectively.  This corresponds to the case of the theorem where
the limit is $\pm \infty$.

Next, we proceed to case (iii).
In this case, $\vopt(t)$ is bounded for sufficiently small $t >0$ because
$|\vopt(t) - (C+t\alpha \Id)\bullet X_\nu(t)|\leq n\nu$ and $(C+\alpha \Id)\bullet X_\nu(t)$ is bounded for sufficiently small
$t>0$.
Therefore, 
there exist three constants $M_1, M_2$, and $\bar{t}>0$ such that
$M_1 < M_2$ and $\bar t >0$ for which
\[
\vopt(t) \in [M_1, M_2]\quad \hbox{if}\ t\in (0, \bar t].
  \]
For the sake of obtaining a contradiction, we assume that $\vopt(t)$ does not have a limit as $t\rightarrow 0$.  Then, 
there exists an infinite sequence $\{t^k\}$ with $\lim_{k\rightarrow\infty} t^k \rightarrow 0$
where $\{\vopt(t^k)\}$ has two distinct accumulation points, $v_1$ and $v_2$, say.
Without loss of generality, we let $v_1 > v_2$ and $z \equiv v_1 - v_2$.

Let $\tilde\nu \equiv z/(6n)$.
By  Step~1, it follows that $(C+t\alpha \Id) \bullet X_{\tilde\nu}(t)$
is a monotone function for sufficiently small $t>0$.  
Furthermore, since $\vopt(t)$ is bounded for sufficiently small $t$, 
(\ref{optvbd}) implies that $(C+t\alpha \Id) \bullet X_{\tilde\nu}(t)$ does not diverge and has a limit as $t\downarrow 0$.
Let us denote by $c_{\tilde\nu}^*$ the limit value, and let $\tilde t >0$ be such that
\begin{equation}\label{z/6a}
|(C+t\alpha \Id) \bullet X_{\tilde\nu}(t) -c_{\tilde\nu}^*|\leq\frac{z}6
\end{equation}
holds for any $t \in (0,\tilde t]$.  
On the other hand,
\begin{equation} \label{z/6b}
|(C+t\alpha \Id) \bullet X_{\tilde\nu}(t) -\vopt(t)|=
(C+t\alpha \Id) \bullet X_{\tilde\nu}(t) -\vopt(t) \leq n\tilde\nu =\frac{z}6
\end{equation}
holds due to (\ref{optvbd}).  Adding \eqref{z/6a}, \eqref{z/6b} and using the triangular inequality, we see that
\[
|c^*_{\tilde\nu}-\vopt(t)|\leq \frac{z}{3}, \ \ \ \hbox{i.e.}, \ \ \ 
c^*_{\tilde\nu}-\frac{1}{3}z\le \vopt(t) \le c^*_{\tilde\nu}+\frac{1}{3}z
\]
holds for any $t \in (0,\tilde t]$.
Together with the fact that $v_1>v_2$ are the two accumulation points of $\{\vopt(t)\}$, the above relation yields
\[
c^*_{\tilde\nu}-\frac{1}{3}z\le v_2< v_1 \le c^*_{\tilde\nu}+\frac{1}{3}z.
\] 
This implies  $z= v_1-v_2\le 2z/3$ and hence $z \leq 0$, which, however, contradicts $z >0$. 
Therefore, the accumulation point of $\vopt(t)$ is unique and the limit of $\vopt(t)$ exists as $t\downarrow 0$.
\qed

\medskip

Now we are ready to prove Theorem \ref{2.2}. Let 
\[
\tilde v(\beta) \equiv \lim_{t\downarrow 0}v(t, t\beta)\ \hbox{for}\ \beta\in[0,\infty),\ \ 
\tilde v(\infty) \equiv \lim_{t\downarrow 0}v(0, t). 
\]
We note that
\[
\av(\theta) = \lim_{t\downarrow 0} v(t\cos\theta, t\sin\theta) = \tilde v(\tan\theta).
\]


\medskip

Theorem \ref{2.2} is a direct consequence of the following theorem.
\begin{theorem}\label{4.2}
If the problem is asymptotically pd-feasible,
then $\tilde v(\beta)$ is a monotone decreasing function in $\beta$ in the interval $[0,+\infty]$ and the following relation holds. 
\[ 
v({\bf D})=\tilde v(\infty)
\leq \tilde v(\beta) \leq
\tilde v(0) =v({\bf P}).
\]
Furthermore, $\tilde v(\beta)$ is a convex function in the interval $[0,\infty)$.
\end{theorem}

\begin{proof}
We first show that $\tilde v$ is a monotone decreasing function in $[0, \infty)$.
Suppose that, by contradiction, monotonicity is violated, namely, there exists
$\beta_1$ and $\beta_2$ such that $\beta_1 < \beta_2$ and $\tilde v(\beta_1) < \tilde v(\beta_2)$.
Let $u = \tilde v(\beta_2) -\tilde v(\beta_1)>0$.
Recall that 
\[ 
\tilde v(\beta) = \lim_{t\rightarrow 0} v(t, t\beta).
\]
We show that for sufficiently small $t$
\[
v(t,t\beta_2) - v(t,t\beta_1) \leq u/2
\]
holds, which contradicts $\tilde v(\beta_2) -\tilde v(\beta_1)=\lim_{t\downarrow0} (v(t,t\beta_2)-v(t,t\beta_1))=u$.
In fact, since $v_{P}(\varepsilon,\eta)=v(\varepsilon,\eta)-\eta(C\bullet \Ip + \varepsilon \Id\bullet \Ip)$ is a monotone decreasing function in $\eta$ (see item~3 of Proposition~\ref{31}),
\[
v(t,t\beta)- t\beta(C\bullet \Ip+ t \Id\bullet \Ip)
\] 
is a monotone decreasing function in $\beta$.
Therefore,
\[
v(t,t\beta_2)-t\beta_2(C\bullet \Ip + t\Id\bullet \Ip)
\leq
v(t,t\beta_1)-t\beta_1(C\bullet \Ip + t\Id\bullet \Ip)
\]
holds.  This implies that, for sufficiently small $t>0$, 
\[
v(t,t\beta_2)-v(t,t\beta_1) \leq t(\beta_2-\beta_1)(C\bullet \Ip + t \Id\bullet \Ip) \leq \frac{u}2
\]
and hence letting $t\rightarrow 0$, we obtain
\[
0 < u = \tilde v (\beta_2)-\tilde v(\beta_1)\leq  \frac{u}2,
\]
contradiction.  

Now we confirm monotonicity at $\beta=\infty$.  Since $\tilde v(\infty) = \lim_{t\downarrow0}v(0,t)$, what we need to show is
$\tilde v(\infty) \leq \tilde v(\beta)$ for any finite $\beta$.  This is confirmed as follows:
\[
\tilde v(\infty) =\lim_{t\downarrow0}v(0,t) \leq \lim_{t\downarrow 0}v(\beta^{-1}t,t) = \lim_{t\downarrow 0}v(t,\beta t) = \tilde v(\beta)\ (\beta >0).
\]
The first inequality is due to the item~2. of Proposition~\ref{31}, and the second equality holds because $\tilde v(\gamma)=\tilde v(k\gamma)$
for any $k>0$, i.e., $\tilde v$ is a homogeneous function.

Now we prove convexity of $\tilde v(\beta)$.  We define the function $v_k$ as
\[
v_k(\beta) \equiv v_P\left(\frac1{k},\frac1{k}\beta \right).
\]
Then it follows for any $\beta \in [0,\infty)$ that
\[
\lim_{k\rightarrow\infty} v_k(\beta ) = \lim_{k\rightarrow\infty}v_P\left(\frac1{k},\frac1{k}\beta\right)=\tilde v(\beta).
\]
Thus, $\{v_k\}$ converges pointwise to $\tilde v$.  By item~3. of Proposition~\ref{31}, $v_k$ is convex on $(0,\infty)$, so it follows from 
 \cite[Theorem~10.8]{rockafellar} that  $\tilde v$ is also a convex function on $(0,\infty)$.
Since $\tilde v(\alpha)$ is monotone increasing on $[0,\infty)$, $\tilde v$ is convex on $[0,\infty)$.  
This completes the proof of the theorem.
\qed\end{proof}

\medskip\noindent
{\bf (Proof of Theorem \ref{2.2})}
We recall that a convex function is continuous over the relative interior of its domain, e.g., \cite[Theorem~10.1]{rockafellar}, so 
the function $\tilde v$ in Theorem~\ref{4.2} is continuous over $(0,\infty)$.
We also recall that $\av(\theta)= \lim_{t\downarrow 0} v(t\cos\theta,t\sin\theta)$.  We have, for $\theta \in [0,\pi/2]$,
\[
\av(\theta) = 
\lim_{t\downarrow 0} v(t\cos\theta,t\sin\theta) = \lim_{t\downarrow 0} v(t ,t\tan\theta) = \tilde v(\tan\theta).
\]
Since $\av(\theta) =\tilde v(\tan \theta)$ and $\tan$ is a strictly monotone increasing function in $\theta$, 
Theorem \ref{2.2} readily follows.

\section{Application to Infeasible Interior-point Algorithms}\label{sec:app}

The analysis in the previous section indicates that the limiting common optimal value of  ${\bf  P}(t\alpha, t\beta)$
and ${\bf  D}(t\alpha, t\beta)$ exists as $t\rightarrow 0$ and the value is between $v({\bf D})$ and $v({\bf P})$.
In this section, we discuss an application to the convergence analysis of  infeasible primal-dual interior-point algorithms.

While the efficiency of  infeasible interior-point algorithms  
is supported by a powerful polynomial-convergence analysis when applied to a primal-dual strongly feasible problems, 
its behavior for singular problems was not clear.
Our analysis leads to a clearer picture about what happens when infeasible interior-point algorithms are applied to 
arbitrary SDP problems.  As indicated in Subsection~\ref{IPM}, we focus on two polynomial-time algorithms by Zhang \cite{Zhang:1998:ESP} and
 Potra and Sheng \cite{ipm:Potra16}, but the idea and the analysis can be applied to many other variants.

Suppose that $\hat X$ is a solution to $A(X) = b$, 
$(\hat S, \hat y)$ is a solution to $S = C - \sum_i A_i y_i$,   
and let 
 \[(X^0, S^0, y^0) \equiv (\hat X + \rho\sin\theta \Ip, \hat S + \rho \cos\theta \Id,0 ),\] where $\theta\in (0, \pi/2)$ and 
$\rho > 0$ is sufficiently large so that $X^0\succ 0$ and $S^0\succ 0$ hold. This is an interior feasible point to
the primal-dual pair ${\bf P}(\rho \cos\theta, \rho \sin\theta)$ and ${\bf D}(\rho \cos\theta, \rho \sin\theta)$,
see (\ref{Pptbd}) and (\ref{Dptbd}).
In the following, we analyze infeasible primal-dual interior-point algorithms started from this point.

For simplicity of notation, we let $\alpha \equiv \cos\theta$ and $\beta \equiv \sin\theta$.
As discussed in Section \ref{3.2.1}, in particular as  stated in Proposition~\ref{StpIPM},
the infeasible primal-dual interior-point algorithms we are considering generate a sequence $(X^k, S^k, y^k)$  of interior feasible points
to the perturbed system
\begin{equation}\label{eq:perturb}
C + t^k \alpha \Id - \sum_i A_i y_i^k= S^k,\ \ \ A (X^k - t^k \beta \Ip)=b,\ \ X^k\succeq 0,\ \ S^k \succeq 0.
\end{equation}
for $t^k \geq 0$.
We define
\begin{equation}\label{modobjs}
(C + t^k \alpha \Id)\bullet X\ \ \ \hbox{and}\ \ \ \sum_i (b_i+ t^k\beta A_i\bullet \Ip) y_i^k
\end{equation}
as the \emph{modified primal objective function} and the  \emph{modified dual objective function}, respectively. 
If  $(X^k, S^k, y^k,t^k)$  is a sequence satisfying \eqref{eq:perturb} for every $k$ and $t^k\downarrow 0$, then it is an \emph{asymptotically pd-feasible sequence} in the sense that $X^k$, $S^k$ satisfy the conic constraints of ${\bf P}$ and ${\bf D}$ and the distance between $(X^k,S^k,y^k)$ and the set of solutions to the \emph{linear constraints} of ${\bf P}$ and ${\bf D}$ goes to $0$ as $k \to \infty$.\footnote{
We note, however, that this does \textbf{not} imply that, say, the distance between $X^k$ and the feasible region of ${\bf P}$ goes to $0$ as $k \to \infty$, even if the feasible region of ${\bf P}$ is not empty. A similar comment applies to $S^k, y^k$ and the feasible region of ${\bf D}$. An instructive example can be seen in \cite[Example~1]{sturm_error_2000}.
}


Now we are ready to describe and prove our first result on infeasible interior-point algorithms. 

\begin{theorem}\label{5.1} \   Suppose that $\hat X$ is a solution to $A(\hat X)=b$,
$(\hat S, \hat y)$ is a solution to $C - \sum_i A_i y_i = S$, and let 
 $(X^0, S^0, y^0) \equiv (\hat X + \rho\sin\theta \Ip,  \hat S + \rho\cos\theta \Id,0 )$, where $\theta\in (0, \pi/2)$ and 
$\rho > 0$ is sufficiently large so that $X^0\succ 0$ and $S^0\succ 0$ hold.  Also, let $t^0\equiv 1$.  Apply the
algorithm Algorithm-B of \cite{Zhang:1998:ESP} or Algorithm 2.1 of \cite{ipm:Potra16} to solve {\bf P} and {\bf D}, and let 
$\{(X^k, S^k, y^k, t^k)\}$ be the generated sequence.  
Then the following statements hold.
\begin{enumerate}

\item $t^k\rightarrow 0$ and  $X^k\bullet S^k \rightarrow 0$ hold if and only if {\bf P} and {\bf D} are asymptotically pd-feasible, namely, the algorithms generate an asymptotically pd-feasible sequence with duality gap converging to zero if and
only if {\bf P} and {\bf D} are asymptotically pd-feasible. See the remark after the proof of the theorem for the behavior 
of the algorithms when {\bf P} and {\bf D} are not asymptotically pd-feasible.

\item  If the problem is asymptotically pd-feasible, then the generated sequence of the modified primal and dual objective 
values (\ref{modobjs}) converges to the value $\av(\theta) \in [v({\bf D}), v({\bf P})]$.
Here, we include the possibility that $\av(\theta) =+\infty$ and $\av(\theta) =-\infty$, interpreting 
them as divergence to $+ \infty$ and $-\infty$,  respectively.

\item In item 2., as $\theta$ gets closer to $0$ the
 limiting modified objective values of the infeasible primal-dual algorithm get closer to the primal optimal value $v({\bf P})$
of the original problem. 
As $\theta$ gets closer to $\pi/2$
the limiting modified objective value gets closer to the dual optimal value 
$v({\bf D})$.

\end{enumerate}
\end{theorem}
\begin{proof}
First, we discuss item~1. If $\{(X^k, S^k, y^k, t^k)\}$ is an asymptotically pd-feasible sequence, then ${\bf P}$ and ${\bf D}$  must  be asymptotically pd-feasible. Next, we take a look at the converse.
In the analysis conducted in \cite{Zhang:1998:ESP,ipm:Potra16}, although both papers assume the existence of a 
solution to \eqref{pd-formulation}, in fact, the existence of a solution is not necessary for showing 
convergence of $t^k$ and $X^k \bullet S^k$ to zero under asymptotic pd-feasibility.
Under asymptotic pd-feasibility, for any $t> 0$ the perturbed problems are strongly feasible.
This is enough for showing $t^k\rightarrow 0$ and $X^k\bullet S^k\rightarrow 0$ in these algorithms.
We give more details of the proof in Appendix~\ref{app2}.

Now we prove items 2 and 3.
The following relations hold at the $k$-th iteration: 
\begin{eqnarray} 
&&(C+t^k \alpha \Id)\bullet X^k - \sum _i (b_i + t^k\beta A_i \bullet \Ip) y_i^k  = X^k\bullet S^k. \label{1215_1} \\
&&v(t^k\alpha, t^k\beta) \in \left[\sum _i (b_i + t^k\beta A_i \bullet \Ip) y_i^k   , (C+t^k \alpha \Id)\bullet X^k \right] \label{1215_2}
\end{eqnarray}
(See also (\ref{223a}) and (\ref{223b}) for the derivation of these relations.)

Then it follows from (\ref{1215_1}), (\ref{1215_2}) and $X^k\bullet S^k \rightarrow 0$
that the sets of accumulation points of $\{(C+t^k \alpha \Id)\bullet X^k\}$, 
$\{v(t^k\alpha, t^k\beta)\}$, and $\{ \sum (b_i +t^k \beta A_i \bullet \Ip) y_i^k\}$
coincide.
Since $t^k\rightarrow 0$, this 
implies that $ v(t^k\alpha, t^k\beta)=v(t^k \cos\theta, t^k \sin\theta) $ converges to $\av(\theta)$. 
Then the sequences of 
the modified objective functions (\ref{modobjs})  also converge to $\av(\theta)$.
\qed\end{proof}

\noindent
{\bf Remark}  When {\bf P} and {\bf D} are not pd-asymptotically feasible, $\lim_{k\rightarrow \infty} t^k$ is positive  for both algorithms \cite{Zhang:1998:ESP,ipm:Potra16}.  
But the behavior of the duality gap $X^k\bullet S^k$ is a bit different.  
In the case of Zhang's algorithm, the sequence of $X^k\bullet S^k$ also converges to a positive value 
as well,  but in the case of Potra and Sheng's algorithm, what we can say is that ${\rm liminf}\ X^k\bullet S^k$ is positive.  
This is because the sequence $X^k\bullet S^k$ is not necessarily monotonically decreasing in Potra and Sheng's algorithm.

\medskip
Now we present the last theorem.
A typical choice of the initial iterate $(X^0, S^0, y^0)$ for  primal-dual infeasible interior-point algorithms is  
$(X^0, S^0, y^0)=(\rho_0 I, \rho_1 I, 0)$ with $\rho_0>0$ and $\rho_1>0$ sufficiently large.  This is different from the one adopted in Theorem \ref{5.1}.  
In concluding this section, we discuss how our results can be adapted to this case.
Let $\hat X$  be a solution to $A(X)=b$.
If we set $\Ip \equiv \rho_0 I - \hat X$ and $\Id \equiv \rho_1 I - C$ with $\rho_0$ and $\rho_1$ sufficiently large so that
$\Ip\succ 0$ and $\Id\succ 0$ hold, 
$(X^0, S^0, y^0)$ is a feasible solution to ${\bf P}(1, 1)$ and ${\bf D}(1, 1)$.
Now, we are ready to apply an argument analogous to the one we developed earlier to derive Theorem \ref{5.1} with this choice of 
$\Ip$ and $\Id$ to obtain the following theorem.  
%
%

\medskip\noindent
\begin{theorem}\label{5.2}\ Let 
$(X^0, S^0, y^0) \equiv (\rho_0 I, \rho_1 I,0 )$, where 
$\rho_0 > 0$ and $\rho_1 > 0$ are sufficiently large so that $\Ip=\rho_0 I- \hat X \succ 0$ and $\Id =\rho_1 I- C \succ 0$ hold, where 
$\hat X$ is a solution to $A(X) = b$.  Apply the algorithm Algorithm-B of \cite{Zhang:1998:ESP} or Algorithm 2.1 of \cite{ipm:Potra16} with the initial iterate
$(X^0, S^0, y^0)$ and $t^0=1$, and let $\{(X^k, S^k, y^k, t^k)\}$ be the generated sequence.  Then the following statements hold:

\begin{enumerate}
\item $t^k\rightarrow 0$ and  $X^k S^k \rightarrow 0$ hold if and only if {\bf P} and {\bf D} are asymptotically pd-feasible,
namely, the algorithm generates an asymptotically pd-feasible sequence with duality gap converging to zero if and
only if {\bf P} and {\bf D} are asymptotically pd-feasible.  
If {\bf P} and {\bf D} are not asymptotically pd-feasible, then 
the same remark after Theorem \ref{5.1} holds.

\item  If the problem is asymptotically-pd feasible, then the generated sequence of the modified primal and dual objective 
values (\ref{modobjs}) converges to a value $\av(\pi/4) \in [v({\bf D}), v({\bf P})]$.
Here, we include the possibility that $\av(\pi/4) =+\infty$ and $\av(\pi/4) =-\infty$, interpreting 
them as divergence to $+ \infty$ and $-\infty$,  respectively.
\end{enumerate}
\end{theorem}

%
%

\medskip


\section{Examples}\label{sec:ex}

In this section, we present three examples with nonzero duality gaps to illustrate Theorems \ref{2.1} and \ref{2.2}.  
The optimal values of {\bf P} and {\bf D} are both finite in Example 1, the optimal value of {\bf P} is finite but {\bf D} is
weakly infeasible in Example 2, and both problems are weakly infeasible in Example 3.  In the latter two cases the duality
gaps are infinity.

\medskip\noindent
{\bf Example 1}

We start with a simple instance with a finite nonzero duality gap taken from Ramana's famous paper \cite{Ramana95anexact}.
The following problem has a duality gap of one.

The problem {\bf D} is
\[
\max\ y_1\ \hbox{s.t.}\ \ \left(\begin{array}{ccc} 1 -y_1 & 0 & 0 \\
0 & -y_2 & -y_1 \\ 
0 & -y_1 & 0
\end{array}\right) \succeq 0.
\]
With that, we have
\[
C=
\left(\begin{array}{ccc} 1 & 0 & 0 \\
0 & 0 & 0 \\ 
0 & 0 & 0
\end{array}\right),\ \ \ 
A_1=
\left(\begin{array}{ccc} 1 & 0 & 0 \\
0 & 0 & 1 \\ 
0 & 1 & 0
\end{array}\right),\ \ \ 
A_2=
\left(\begin{array}{ccc} 0 & 0 & 0 \\
0 & 1 & 0 \\ 
0 & 0 & 0
\end{array}\right),\ \ \ b_1=1.
\]
The optimal value $v({\bf D})= 0$ for this problem, since $y_1=0$ is the only possible value for the lower-right $2\times 2$ submatrix to be 
positive semidefinite.

The associated primal {\bf P} is
\[
\min\ x_{11}\ \ \hbox{s.t.}\ x_{11}+ 2 x_{23} = 1, \ x_{22}=0,\ \left(\begin{array}{ccc} x_{11} & x_{12} & x_{13}\\ x_{12} & x_{22} & x_{23}\\
x_{13} & x_{23} & x_{33}\end{array}\right)\succeq 0.
\]
The optimal value $v({\bf P})=1$ for this problem, since $x_{23}=0$ must hold for positive semidefiniteness of the lower-right $2\times 2$
submatrix, which drives $x_{11}$ to be 1.

Now we consider the problem ${\bf D}(\varepsilon,\eta)$
\[
\max\ (1+\eta) y_1 + \eta y_2\ \ \hbox{s.t.}\ \left(\begin{array}{ccc} 1+\varepsilon -y_1 & 0 & 0 \\
0 & \varepsilon -y_2 & -y_1 \\ 
0 & -y_1 & \varepsilon
\end{array}\right) \succeq 0.
\]
This is equivalent to 
\[
\max\ (1+\eta) y_1 + \eta y_2\ \ \hbox{s.t.}\ 1+\varepsilon-y_1 \geq 0,\ \ \ \varepsilon(\varepsilon - y_2) - y_1^2 \geq 0.
\]

Since the objective is linear, there is an optimal solution 
such that at least one of the  inequality constraints is active.  
Taking into account that the second constraint is quadratic, we analyze the following three subproblems, and take
the maximum of them.
\begin{eqnarray*}
&&
{\rm (Case\ 1)}\ 
\max\ (1+\eta) y_1 + \eta y_2\ \ \hbox{s.t.}\ 1+\varepsilon-y_1 = 0,\ \ \ \varepsilon(\varepsilon - y_2) - y_1^2 \geq 0.\\
&&
{\rm (Case\ 2)}\ 
\max\ (1+\eta) y_1 + \eta y_2\ \ \hbox{s.t.}\ 1+\varepsilon-y_1 \geq 0,\ \ \ y_1=\sqrt{\varepsilon(\varepsilon - y_2)}.\\
&&
{\rm (Case\ 3)}\ 
\max\ (1+\eta) y_1 + \eta y_2\ \ \hbox{s.t.}\ 1+\varepsilon-y_1 \geq 0,\ \ \ y_1=-\sqrt{\varepsilon(\varepsilon - y_2) }.
\end{eqnarray*}

\noindent
(Case 1)

In this case, the second constraint yields
\[
\varepsilon - \frac{(1+\varepsilon)^2}{\varepsilon} \geq y_2.
\]
Together with $y_1=1+\varepsilon$, the problem reduces to a  linear program, and it follows that the maximum is
\[
v_1(\varepsilon,\eta)\equiv(1+\eta)(1+\varepsilon)+\eta\varepsilon - \frac{\eta(1+\varepsilon)^2}{\varepsilon}.
\]

\noindent
(Case 2)

Under this condition, the objective function is written as
\[
f(y_2)\equiv (1+\eta)\sqrt{\varepsilon(\varepsilon-y_2)}+\eta y_2.
\]
By computing the derivative, we see that the function takes the unique maximum at
\begin{equation}\label{maxi}
y_2 = \varepsilon -\frac{\varepsilon(1+\eta)^2}{4\eta^2}
\end{equation}
and
\begin{equation}\label{1}
\sqrt{\varepsilon(\varepsilon-y_2)} = \frac{\varepsilon(1+\eta)}{2\eta}.
\end{equation}
Then, we see that
\begin{equation} \label{2}
f(y_2) = \varepsilon\eta + \frac{\varepsilon}{4\eta}(1+\eta)^2.
\end{equation}
But we should recall that this maximum is obtained by ignoring the constraint
\[
1+ \varepsilon - y_1 = 1+\varepsilon - \sqrt{\varepsilon(\varepsilon - y_2)}\geq0.
\]
By substituting (\ref{maxi}) and  (\ref{1}) into this constraint, (\ref{2}) is the maximum only if
\begin{equation}\label{3}
1+\varepsilon - \frac{\varepsilon(1+\eta)}{2\eta} \geq 0, \hbox{or,\ equivalently}, \ \frac{2\eta}{1-\eta} \geq \varepsilon
\end{equation}
is satisfied.

If (\ref{3}) does not hold, then, the maximum of $f(y_2)$ is taken at the boundary of the constraint
$1+\varepsilon -y_1 \geq 0$, i.e., $y_2$ satisfying the condition
\[
1+\varepsilon = \sqrt{\varepsilon(\varepsilon-y_2)}.
\]
Solving this equation with respect to $y_2$, we obtain
\[
y_2=-2-\frac{1}{\varepsilon},\ y_1 = 1+\varepsilon,\ \ f(y_2) = (1+\eta)(1+\varepsilon)-\eta\left(2+\frac1{\varepsilon}\right).
\]
In summary, the maximum value in (Case 2) is as follows:
\begin{eqnarray}
&&v_{2}(\varepsilon,\eta)\equiv\varepsilon\eta + \frac{\varepsilon}{4\eta}(1+\eta)^2\ \ \hbox{if}\ \frac{2\eta}{1-\eta} \geq \varepsilon, \\
&&v_2(\varepsilon,\eta)\equiv
(1+\eta)(1+\varepsilon)-\eta\left(2+\frac1{\varepsilon}\right)\ \ \hbox{if}\ \frac{2\eta}{1-\eta} \leq \varepsilon
\end{eqnarray}

\medskip\noindent
(Case 3)

In this case, $1+\varepsilon - y_1 \geq 0$ holds trivially.
Therefore, the maximization problem in this case is 
\[
\max -(1+\eta)\sqrt{\varepsilon(\varepsilon-y_2)}+\eta y_2.
\]
under the condition that $y_2 \leq \varepsilon$.  The function is 
monotone increasing, so that the maximum is attained when $y_2 = \varepsilon$ and
the maximum value is
\[ 
v_3(\varepsilon,\eta)\equiv\eta\varepsilon.
\]
\medskip 
Now we are ready to combine the three results to complete the evaluation of $\tilde v$ and $\av$.  
By letting $\varepsilon = t\alpha$, $\eta=t\beta$ with $t> 0$ and
letting $t\downarrow 0$, we see that

(Case 1) $\lim_{t\downarrow 0}v_1(t\alpha,t\beta)=0$.

(Case 2) $\lim_{t\downarrow 0}v_2(t\alpha,t\beta)=\frac{\alpha}{4\beta}$ if $\frac{\beta}{\alpha}\geq \frac12$,\ \ 
$\lim_{t\downarrow 0}v_2(t\alpha,t\beta)=1-\frac{\beta}{\alpha}$ if $\frac{\beta}{\alpha}\leq \frac12$

(Case 3) $\lim_{t\downarrow 0}v_3(t\alpha,t\beta)= 0$.

The maximum among the three corresponds to $\tilde v$.  Comparing the three, we see that (Case 2) always is the 
maximum.  This means
\[
\tilde v(\beta)=1 - \beta\ (\beta \in [0,\frac12]),\ \ \ \tilde v(\beta)=\frac1{4\beta}\ (\beta\in[\frac12, \infty)),\ 
\ \tilde v(\infty)=0.
\]

\medskip\noindent
{\bf Example 2}

The next example is such that {\bf D} is weakly infeasible but {\bf P} is weakly feasible and has a finite optimal value.

The problem {\bf D} is
\[
\max\ -y_1\ \ \hbox{s.t.}\ \left(\begin{array}{ccc} y_2 & 0 & 1 \\
0 & y_1 & 0 \\ 
1 & 0 & 0
\end{array}\right) \succeq 0.
\]
\[
C=
\left(\begin{array}{ccc} 0 & 0 & 1 \\
0 & 0 & 0 \\ 
1 & 0 & 0
\end{array}\right),\ \ \ 
A_1=
\left(\begin{array}{ccc} 0 & 0 & 0 \\
0 & -1 & 0 \\ 
0 & 0 & 0
\end{array}\right),\ \ \ 
A_2=
\left(\begin{array}{ccc} -1 & 0 & 0 \\
0 & 0 & 0 \\ 
0 & 0 & 0
\end{array}\right),
\ \ \ b_1=-1.
\]
This system is weakly infeasible, so $v({\bf D})=-\infty$.

The associated primal {\bf P} is
\[
\min\ 2x_{13}\ \ \hbox{s.t.}\ x_{11} = 0, \ x_{22}=1,\ \left(\begin{array}{ccc} x_{11} & x_{12} & x_{13}\\ x_{12} & x_{22} & x_{23}\\
x_{13} & x_{23} & x_{33}\end{array}\right)\succeq 0.
\]
The optimal value $v({\bf P})=0$ for this problem, since $x_{13}=0$ must hold for feasibility.

Now we consider the problem ${\bf D}(\varepsilon,\eta)$
\[
\max\ -(1+\eta) y_1 - \eta y_2\ \ \hbox{s.t.}\ \left(\begin{array}{ccc} y_2+\varepsilon & 0 & 1 \\
0 & y_1+\varepsilon & 0 \\ 
1 & 0 & \varepsilon
\end{array}\right) \succeq 0.
\]
It follows that 
\[
y_1 \geq -\varepsilon,\ \ \ y_2 \geq \frac{1-\varepsilon^2}{\varepsilon}.
\]
Therefore, we see that the maximum value is
\[
v(\varepsilon,\eta) = (1+\eta)\varepsilon - \frac{1-\varepsilon^2}{\varepsilon}\eta.
\]
Now we are ready to evaluate $\tilde v$ and $\av$.  
By letting $\varepsilon = t\alpha$, $\eta=t\beta$ with $t> 0$ and
letting $t\downarrow 0$, we see that
\[
\lim_{t\downarrow0} v(t\alpha,t\beta)=-\frac{\beta}{\alpha}.
\]
and
\[
\tilde v(\beta)=-\beta\ (\beta \in [0,\infty]).
\]

Finally, we deal with a pathological case where both primal and dual are weakly infeasible.

\medskip\noindent
{\bf Example 3}

The problem {\bf D} is
\[
\max\ y_1\ \ \hbox{s.t.}\ \left(\begin{array}{ccc} y_2 & 0 & 1+\frac12y_1 \\
0 & 1+y_1 & 0 \\ 
1+\frac12y_1 & 0 & 0
\end{array}\right) \succeq 0.
\]
\[
C=
\left(\begin{array}{ccc} 0 & 0 & 1 \\
0 & 1 & 0 \\ 
1 & 0 & 0
\end{array}\right),\ \ \ 
A_1=
\left(\begin{array}{ccc} 0 & 0 & -\frac12 \\
0 & -1 & 0 \\ 
-\frac12 & 0 & 0
\end{array}\right),\ \ \ 
A_2=
\left(\begin{array}{ccc} -1 & 0 & 0 \\
0 & 0 & 0 \\ 
0 & 0 & 0
\end{array}\right),\ \ \
b_1=1.
\]
The optimal value $v({\bf D})= -\infty$ for this problem, since $y_1=-2$ should hold for feasibility, but
then the (2,2) element becomes $-1$ and, therefore,  the matrix cannot be feasible.  By letting $y_2$ large and $y_1=0$, we confirm
the problem is weakly infeasible.

The associated primal {\bf P} is
\[
\min\ 2x_{13}+x_{22}\ \ \hbox{s.t.}\ x_{13}+ x_{22} = -1, \ x_{11}=0,\ \left(\begin{array}{ccc} x_{11} & x_{12} & x_{13}\\ x_{12} & x_{22} & x_{23}\\
x_{13} & x_{23} & x_{33}\end{array}\right)\succeq 0.
\]
This problem is weakly infeasible.

Now we consider the problem ${\bf D}(\varepsilon,\eta)$
\[
\max\ (1-\eta) y_1 - \eta y_2\ \ \hbox{s.t.}\ \left(\begin{array}{ccc} \varepsilon +y_1 & 0 & 1+\frac12 y_1 \\
0 & 1+\varepsilon +y_1 & 0 \\ 
1+\frac12 y_1 & 0 & \varepsilon
\end{array}\right) \succeq 0.
\]
This is equivalent to 
\[
\max\ (1-\eta) y_1 - \eta y_2\ \ \hbox{s.t.}\ \varepsilon+y_2 \geq 0,  \ \ \varepsilon(\varepsilon + y_2) - (1+\frac12y_1)^2 \geq 0,  \ \ \ 1+\varepsilon+y_1 \geq 0.
\]

Since the objective is linear, there is an optimal solution 
such that at least one of the  inequality constraints is active.  
Taking into account that the second constraint is quadratic, we analyze the following three subproblems and take
the maximum of them.

\begin{eqnarray*}
{\rm (Case\ 1)}&& 
\max\ (1-\eta) y_1 - \eta y_2\ \ \hbox{s.t.}\ \varepsilon+y_2 = 0,  \ \varepsilon(\varepsilon + y_2) - (1+\frac12y_1)^2 \geq 0,  \\ 
&& 1+\varepsilon+y_1 \geq 0.
\\
{\rm (Case\ 2)}&&
\max\ (1-\eta) y_1 - \eta y_2\ \ \hbox{s.t.}\ \varepsilon+y_2 \geq 0,  \ \varepsilon(\varepsilon + y_2) - (1+\frac12y_1)^2 = 0,  \\
&& 1+\varepsilon+y_1 \geq 0.
\\
{\rm (Case\ 3)}&&
\max\ (1-\eta) y_1 - \eta y_2\ \ \hbox{s.t.}\ \varepsilon+y_2 \geq 0,  \ \varepsilon(\varepsilon + y_2) - (1+\frac12y_1)^2 \geq 0,  \\
&& 1+\varepsilon+y_1 = 0.
\end{eqnarray*}

\medskip
\noindent
(Case 1)

In this case, we have $y_2=-\varepsilon$, $y_1=-2$.  Then the third constraint becomes $\varepsilon -1 \geq 0$.
Since we are interested in the situation where $\varepsilon$ is approaching zero, we may exclude this case.

\noindent
(Case 2)

In this case, we have
\[
\varepsilon(\varepsilon+y_2)=(1+\frac12y_1)^2.
\]
This implies that
\[
y_1 = 2(-1\pm\sqrt{\varepsilon(\varepsilon+y_2)).}
\]
Since the condition $1+\varepsilon+y_1 \geq 0$ yields
\[
\pm\sqrt{\varepsilon(\varepsilon+y_2)} \geq 1 - \varepsilon,
\]
choosing `-' sign is not compatible with our analysis since we are interested in 
the case where $\varepsilon$ is close to zero.
Therefore, we pick `+' sign, and seek for the maximum of the objective function
\[
	2(1-\eta)(-1+\sqrt{\varepsilon(\varepsilon+y_2)}-\eta y_2.
\]
By differentiation, we see that the function attains its maximum at
\[
y_1=2\left(-1+\frac{\varepsilon(1-\eta)}{\eta}\right),\ \ \ y_2=\frac{\varepsilon}{\eta^2}(1-2\eta).
\]
We see that the first constraint is always satisfied at the maximum.
The third constraint $1+y_1 + \varepsilon \geq 0$ is satisfied if
\[
\frac{\varepsilon}{\eta} \geq \frac{1+\varepsilon}2. 
\]
If this condition is not satisfied, then $1+y_1+\varepsilon=0$ holds at the maximum, so, we can leave the analysis
to the third case. 
Substituting $y_1, y_2$ to the objective, we conclude that, if $\varepsilon/\eta \geq 1$, then, the maximum is
\[
v_2(\varepsilon,\eta) \equiv 2(1-\eta)\left(-1-\varepsilon+\frac{\varepsilon}{\eta}\right)-\frac{\varepsilon}{\eta}+2\varepsilon,
\]
and if the aforementioned condition is not satisfied, then, we can leave the analysis to the third case below.

\medskip\noindent
(Case 3)

We have $y_1=-1-\varepsilon$.  After simple manipulation, we see that other two inequalities are satisfied iff
\[
y_2\geq\frac1{\varepsilon}\left(\frac{1-\varepsilon}2\right)^2 - \varepsilon. 
\]
Therefore, the maximum is 
\[
v_3(\varepsilon,\eta)\equiv-(1-\eta)(1+\varepsilon)-\frac{\eta}{\varepsilon}\left(\frac{1-\varepsilon}2\right)^2 + \varepsilon\eta.
\]

\medskip 
Now we are ready to combine the three results to complete evaluation of $\tilde v$ and $\av$.  
By letting $\varepsilon = t\alpha$, $\eta=t\beta$ with $t> 0$ and
letting $t\downarrow 0$, we see that

(Case 1) Cannot occur.

(Case 2) $\lim_{t\downarrow 0}v_2(t\alpha,t\beta)=-2+\frac{\alpha}{\beta}$ if $\frac{\alpha}{\beta}\geq \frac12$.

(Case 3) $\lim_{t\downarrow 0}v_3(t\alpha,t\beta)= -1-\frac14\frac{\beta}{\alpha}$.

The maximum between the latter two corresponds to $\tilde v$.  Thus, we obtain that
\[
\tilde v(\beta)=-2+\frac1{\beta}\ (\beta \in [0,2]),\ \ \ \tilde v(\beta)=-1-\frac{\beta}4\ (\beta\in[2, \infty]),
\]
where we used the convention $1/0 =\infty$.

\section{Concluding Discussion}\label{sec:conc}
In this paper, we developed a perturbation analysis for singular primal-dual semidefinite programs.
We assumed that primal and dual problems are asymptotically feasible and added positive definite perturbations 
to recover strong feasibility.   A major innovation was that we considered perturbations of 
primal and dual problems simultaneously. 
It was shown that the primal-dual common optimal value of the perturbed problem has a directional limit when
the perturbation is reduced to zero along a line. 
Representing the direction of approach with an angle $\theta$ between 0 and $\pi/2$, where the former and latter corresponds to 
the dual-only perturbation and the primal-only perturbation, respectively, we demonstrated 
that the limiting objective value is a monotone decreasing function in $\theta$ which takes the primal optimal value $v({\bf P})$ 
at $\theta=0$ and the dual optimal value $v({\bf D})$ at $\theta = \pi/2$.
Based on this result, we could show that 
the modified objective values of the two infeasible primal-dual interior-point algorithms by Zhang and by Potra and Sheng
converge to a value between the optimal values of {\bf P} and {\bf D}.  The modified primal and dual objective functions are easily 
computed from the current iterate.  { The development of analogous results
for homogeneous self-dual interior-point algorithms
and the design of robust infeasible primal-dual interior-point algorithms reflecting the theory developed in this paper
are interesting further research topics to explore}.

\section*{Acknowledgements}
	We thank the referees and the associate editor for their  comments, which helped to greatly improve the paper.

\section*{Appendix A: Outline of a Proof of the Existence and Analyticity of the Path ${\cal C} =\{w_\nu(t)| \ 0 <t <\infty \}$}\label{app}

Let
\begin{eqnarray*}
&&\phi_1(X, S, y) = X^{1/2}S X^{1/2}-\nu I,\ \ \ \phi_2(X, S, y)= C  - \sum_i A_i y_i- S ,\\
&&\phi_3(X, S, y) = \left(\begin{array}{c}A_1\bullet X  -b_1\\ \vdots\\
A_m\bullet X  -b_m \end{array}\right).
\end{eqnarray*}
Then, $w_\nu(t)$ is a unique solution to 
\[
\Phi(X,S,y,t) \equiv 
\left(\begin{array}{c}\phi_1(X,S,y)\\ \phi_2(X,S,y) +t\alpha I \\ \phi_3(X,S,y)-t\beta I\end{array}\right) = 0.
\]
$\Phi$ is an analytic mapping from $\{(X, S, y, t) \in {\cal S}_{++}^n\times {\cal S}_{++}^n\times \Re^{(m+1)}\}$ to ${\cal S}_{++}^n\times {\cal S}^n\times \Re^m$,
where ${\cal S}_{++}$ is the set of symmetric positive definite matrices.
Therefore, in order to show the existence and analyticity of the path with the help of the analytic version of the implicit function theorem, it is enough to confirm that the rank of the Jacobian matrix of $\Phi$ is $n(n+1)+m$.  
To this end, we show that the Jacobian matrix of the mapping
\[
\left(\begin{array}{c}\phi_1(X,S,y)\\ \phi_2(X,S,y) \\ \phi_3(X,S,y)\end{array}\right)
\]
is nonsingular.  Indeed it is essentially shown in Theorem 2.4 of \cite{Monteiro:1999:PCN} that
the Jacobian matrix is nonsingular if $\phi_1=0$, i.e., $XS =\nu I$.
(See also the note following the theorem.)

\section*{Appendix B: Outline of a Proof of Item 1 of Theorems \ref{5.1} and \ref{5.2}}\label{app2}

First, we observe that if either 
{\textbf{P}} or {\textbf{D}} is strongly infeasible, it is not possible to find $\{t^k\}$, $X^k$, and $S^k$ satisfying $t^k \to 0$ and $X^k \bullet S^k \to 0$.
It remains to show the converse, that is, assuming that {\bf P} and {\bf D} are
asymptotically pd-feasible, we need to show that $\{t^k\}$ generated by the algorithm converges to $0$ and $X^k \bullet S^k \to 0$. 
We provide an explanation for Zhang's algorithm (Algorithm B in Section 6.2 of \cite{Zhang:1998:ESP}). A similar argument also holds for Potra and Sheng's algorithm.  

As was explained in Subsection \ref{3.2.1}, the algorithm generates a sequence $\{(X^k, S^k, y^k, t^k)\}$ where
$X^k$ and $(S^k,y^k)$ are feasible solutions to ${\bf P}(t^k\alpha, t^k\beta)$ and ${\bf D}(t^k\alpha, t^k\beta)$,
respectively, $X^k \succ 0$, $S^k \succ 0$ and $t^k$ is a monotonically decreasing sequence with $t^0 = 1$.   The matrices $I_p$ and $I_d$ used to define
 ${\bf P}(t\alpha, t\beta)$ and ${\bf D}(t\alpha, t\beta)$ are determined by the initial values.
 We define ${\cal P}(t)$ and ${\cal D}(t)$ as the feasible regions of ${\bf P}(t\alpha, t\beta)$ and ${\bf D}(t\alpha, t\beta)$,
 respectively.
If the problems are asymptotically pd-feasible,
then for any $t>0$, ${\bf P}(t\alpha, t\beta)$ and ${\bf D}(t\alpha, t\beta)$ are strongly feasible. For the sake of obtaining a  contradiction, 
suppose that
$t^k$ has a positive limit $t^*>0$.  The iterates $(X^k, S^k, y^k, t^k)$ are confined to $\Omega$, where
\[
\Omega\equiv \{(X,S,y,t)\ | \ X \in {\cal P}(t),\ (S,y)\in {\cal D}(t),\ t\in [t^*,1], X\bullet S \leq X^0\bullet S^0\}.
\]
By using the facts that ${\bf P}(t\alpha, t\beta)$ and ${\bf D}(t\alpha, t\beta)$ are strongly feasible
for any $t>0$ and that the difference of the objective functions of ${\cal P}(t)$ and ${\cal D}(t)$, which is
nothing but $X\bullet S$, is bounded in $\Omega$, we can show that $\Omega$ is compact.  Therefore, $\{(X^k,S^k,y^k,t^k)\}$ has an accumulation point
$(X^*, S^*, y^*, t^*)$.  
The point $(X^*, S^*, y^*)$ is in the neighborhood of the central path employed by the algorithm 
(see (4.2)-(4.6) of \cite{Zhang:1998:ESP}). 
In a sufficiently small neighborhood $\Omega' \subseteq \{(X,S,y)|(X,S,y,t) \in \Omega\}$ of $(X^*,S^*,y^*)$, the  
search direction is well-defined and 
is a continuous function of $(X,S,y)$. 
Therefore, the norm of the search direction is  bounded over $\Omega'$.
This enables us to show that the step $s^k$ in \eqref{catA} is bounded away from zero uniformly if $(X^k, S^k, y^k)\in \Omega'$.
Then there exists $\zeta > 0$ such that $s^k> \zeta$ for all $k$ sufficiently large.  This contradicts that 
$t^k\rightarrow t^*$, because, in view of \eqref{recur_t}, we have $t^* \leq t^{k+1} = (1- s^k)t^k < (1-\zeta)t^k$, but this cannot hold for $k$ sufficiently large thus leading to a contradiction.

Next, we show that $t^k\rightarrow 0$ yields $X^k\bullet S^k \rightarrow 0$.  The stepsize $s^k$ is controlled in such a way that
\begin{equation}\label{lastone}
\frac{X^{k+1}\bullet S^{k+1}}{X^k\bullet S^k} \leq 1 - \eta s^k
\end{equation}
holds in the algorithm, where $\eta\in (0,1)$ is a constant.  (To see this, we associate the stepsize $s^k$ to 
$\alpha_+$ in Algorithm-B of \cite{Zhang:1998:ESP}.  From the definition of $\alpha_+$ in the bottom line of p.368 and (2.6) of
\cite{Zhang:1998:ESP}, we see that \eqref{lastone} holds with $\eta=(1-\sigma)/2$.)
There are two possible cases.   The first case is $s^k=1$ for some $k=\hat k$, say.  
In that case, after the $\hat k$th iteration, the algorithm
becomes a feasible path following method, and $X^k\bullet S^k$ converges to zero following a standard argument, see Algorithm-A in \cite{Zhang:1998:ESP}.
In the second case, $s^k< 1$ for all $k$.  Since $\lim_{k\rightarrow\infty}t^k=\lim_{k\rightarrow\infty} \prod_{l=0}^k (1-s^l) =0$ yields
$\lim_{k\rightarrow\infty} \prod_{l=0}^k (1-\eta s^l) =0$, we obtain
\[
\lim_{k\rightarrow\infty} X^k\bullet S^k \leq \lim_{k\rightarrow\infty} \prod_{l=0}^k (1-\eta s^l) X^0\bullet S^0 =0. 
\]


%
%

\bibliographystyle{spmpsci}      
\bibliography{regularization_mp_archives2a}   

%
%

\end{document}